\documentclass[12pt,letterpaper %,draft
]{amsart}
\usepackage{tikz}
\usepackage{amssymb}

\title% [Strichartz on de Sitter]
{Strichartz estimates on asymptotically
  de Sitter spaces}
\author{Dean Baskin}
\date{June 8, 2012}

\newtheorem{thm}{Theorem} 
\newtheorem{lem}[thm]{Lemma}
\newtheorem{prop}[thm]{Proposition}

\theoremstyle{definition} 
\newtheorem{defn}[thm]{Definition}

\theoremstyle{remark} 
\newtheorem{remark}[thm]{Remark}

\numberwithin{thm}{section}
\numberwithin{equation}{section}

\newcommand{\diag}{\operatorname{diag}}

\newcommand{\grad}{\nabla}
\newcommand{\lap}{\Delta}
\newcommand{\naturals}{\mathbb{N}}
\newcommand{\norm}[2][]{\left\| #2 \right\|_{#1}} 
 
\newcommand{\pd}[1][]{\partial_{#1}}
\newcommand{\reals}{\mathbb{R}}
\newcommand{\sphere}{\mathbb{S}}

\newcommand{\differential}[1]{\,d#1}
\renewcommand{\dh}{\differential{h}}
\newcommand{\dk}{\differential{k}}
\newcommand{\ds}{\differential{s}}
\newcommand{\dt}{\differential{t}}
\newcommand{\dT}{\differential{T}}
\newcommand{\dv}{\differential{v}}
\newcommand{\dx}{\differential{x}}
\newcommand{\dy}{\differential{y}}

\newcommand{\deta}{\differential{\eta}}
\newcommand{\dsigma}{\differential{\sigma}}
\newcommand{\dtau}{\differential{\tau}}
\newcommand{\dtheta}{\differential{\theta}}
\newcommand{\dzeta}{\differential{\zeta}}

\newcommand{\espaces}[2][t]{H_{E,#2}(#1)}
\newcommand{\espace}[1][t]{H_{E}(#1)}

\newcommand{\F}{\mathcal{F}}

\newcommand{\xt}{\tilde{x}}
\newcommand{\yt}{\tilde{y}}

\newcommand{\dblspace}{\widetilde{X}_{0}^{2}}
\newcommand{\dblzero}{X_{0}^{2}}
\newcommand{\frontface}{\operatorname{ff}}
\newcommand{\newface}{\operatorname{lcf}}
\newcommand{\LC}{\operatorname{LC}}

\newcommand{\leftface}{\operatorname{lf}}

\newcommand{\rightface}{\operatorname{rf}}

\newcommand{\LDfull}[4][m]{I^{#1}_{#2}\left( #3 ; #4\right)}
\newcommand{\LDzero}[3][m]{\LDfull[#1]{#2}{\dblzero}{#3}}
\newcommand{\phgLD}[3][m]{\mathcal{A}_{\operatorname{phg}}^{#2}\LDfull[#1]{}{\dblspace}{#3}}
\newcommand{\phg}[2]{\mathcal{A}_{\operatorname{phg}}^{#1}(\dblspace)}
\newcommand{\Wzero}{{}^{0}W}

\begin{document}

\begin{abstract}
  In this article we prove a family of local (in time) weighted
  Strichartz estimates with derivative losses for the Klein-Gordon
  equation on asymptotically de Sitter spaces and provide a heuristic
  argument for the non-existence of a global dispersive estimate on
  these spaces.  The weights in the estimates depend on the mass
  parameter and disappear in the ``large mass'' regime.  We also
  provide an application of these estimates to establish small-data
  global existence for a class of semilinear equations on these
  spaces.
\end{abstract}
% \thanks{The author is very grateful to Rafe Mazzeo and Andr{\'a}s Vasy
%   for countless helpful discussions and for supervising the thesis on
%   which much of this research is based.  The author is also grateful
%   to Jared Wunsch, Austin Ford, Jeremy Marzuola, and Terence Tao for
%   helpful conversations, and to an anonymous referee for numerous
%   suggestions to improve this paper.  This research was partially
%   supported by NSF grants DMS-0801226 and DMS-1103436.}
% \subjclass{Primary 58J45; Secondary 35L05, 83C99}
% \keywords{Strichartz, de Sitter}

\maketitle

\section{Introduction}
\label{sec:introduction}

In this paper we prove a family of local (in time) weighted
Strich\-artz estimates with derivative losses for the Klein-Gordon
equation on asymptotically de Sitter spaces, with the constants
depending explicitly on the length of the time interval.  We also
provide a heuristic argument for the non-existence of a global
dispersive estimate on these spaces.  As an application we establish
small-data global well-posedness for a class of defocusing semilinear
wave equations in this context.

The main novelty of this paper is twofold.  The first is that local
Strich\-artz estimates with loss follow from the Fourier integral
operator (i.e., Lagrangian distribution) representation of the
propagator near the diagonal.  In particular, though we do not show it
here, our method extends to prove dispersive-type estimates near the
diagonal for a somewhat larger class of Lagrangian distributions.  The
second is that our method works even in the presence of time-dependent
metrics (see also Tataru \cite{tataru:2001a} or Smith
\cite{smith:2006} for wave-packet proofs of Strichartz estimates in a
time-dependent setting).

Strichartz estimates are mixed $L^{p}$ (in time) and $L^{q}$ (in
space) estimates that provide a measure of dispersion for the wave and
Klein-Gordon equations.  A version of these estimates was originally
discovered by Strichartz \cite{Strichartz:1977} and first appeared in
their modern form in works of Ginibre and Velo \cite{Ginibre:1984} (for
the Schr{\"o}dinger equation); Kapitanski{\u\i}
\cite{Kapitanskii:1991} and Mockenhaupt, Seeger, and Sogge
\cite{Mockenhaupt:1993} (for the wave equation); and Pecher
\cite{Pecher:1984} (for the Klein-Gordon equation).  These estimates
have been instrumental in proving well-posedness for various
semilinear dispersive equations.  In the context of general
relativity, Marzuola \emph{et al.}\ \cite{Marzuola:2010}
and Tohaneanu \cite{Tohaneanu:2012} established Strichartz estimates
for the static Schwarzschild and the rotating Kerr black hole
backgrounds, respectively.  Yagdjian and Galstian
\cite{Yagdjian:2009a} proved a family of Strichartz-type estimates on
a model of de Sitter space.

We now describe Strichartz estimates for the wave and Klein-Gordon
equations on Minkowski space $\reals \times \reals^{n}$.  For
allowable exponents $(p,q,s)$, a solution $u$ of $\Box u = 0$ or $\Box
u + u = 0$ satisfies the following estimate:
\begin{align*}
  \left( \int_{0}^{T}\norm[W^{1-s,q}]{u(t,\cdot)}^{p}\right)^{1/p} &+
  \left(
    \int_{0}^{T}\norm[W^{-s,q}]{\pd[t]u(t,\cdot)}^{p}\right)^{1/p} \\
  &\quad\quad\quad\quad \lesssim \norm[H^{1}]{u(0,\cdot)} + \norm[L^{2}]{(\pd[t]u)(0,\cdot)}
\end{align*}
A similar estimate holds for the inhomogeneous equation.  If $T <
\infty$, then the estimate is \emph{local}, while if $T=\infty$ it is
\emph{global}.  The allowable exponents $(p,q,s)$ must satisfy two
conditions: the \emph{admissibility} condition and the \emph{scaling}
condition.  For the wave equation the conditions are as follows:
\begin{align*}
  \frac{2}{p} + \frac{n-1}{q} &\leq \frac{n-1}{2} \tag{wave admissibility} \\
  \frac{1}{p} + \frac{n}{q} &= \frac{n}{2} - s \tag{scaling}
\end{align*}
With the Klein-Gordon equation, the scaling condition is unchanged,
while the admissibility condition is less restrictive:
\begin{equation*}
  \frac{2}{p} + \frac{n}{q} \leq \frac{n}{2} \tag{KG admissibility}
\end{equation*}
If $s$ is larger than the value indicated by the scaling condition,
the Strichartz estimates are said to have a loss of derivatives.  The
difference seen in the Strichartz exponents between the wave and
Klein-Gordon equation is largely a long-time effect.  As seen below in
Theorem~\ref{thm:strichartz-homog}, we obtain analogs of the wave
exponents and not the Klein-Gordon exponents.  This effect is due to
the lack of a global dispersive estimate in our setting.

The asymptotically de Sitter spaces we consider in this paper are
diffeomorphic to $\reals \times Y$ for some compact $n$-dimensional
manifold $Y$ and are equipped with metrics having the following form
near $t = +\infty$:
\begin{equation*}
  -\dt^{2} + e^{2t}h(e^{-t}, y, \dy)
\end{equation*}
Here $h$ is a smooth (as a function of $e^{-t}$) family of Riemannian
metrics on $Y$.  In other words, we consider metrics that are
asymptotically de Sitter in a global sense, i.e., they have a
spacelike infinity.  Note that the de Sitter-Schwarzschild and de
Sitter-Kerr black hole spacetimes do not fall into this class.  A more
precise description of the class of metrics considered is given in
Section~\ref{sec:asympt-de-sitt}.

The main result of this paper is the following Strichartz-type
estimate for the operator $P(\lambda) = \Box_{g} + \lambda$ (a
dictionary of function spaces is found in Section~\ref{sec:notation}):
\begin{thm}
  \label{thm:strichartz-homog}
  Suppose that $(X,g)$ is an asymptotically de Sitter space, $t_{0}$
  is sufficiently large, $\lambda > 0$ and $\alpha$ is such that $0
  \leq \alpha < \sqrt{\lambda}$ if $\lambda \leq \frac{n^{2}}{4}$ and
  $\alpha = \frac{n}{2}$ if $\lambda > \frac{n^{2}}{4}$.  Suppose
  further that $\epsilon > 0$ and $u$ satisfies the homogeneous
  initial value problem:
  \begin{align*}
    P(\lambda) u &= 0 \\
    (u,\pd[t]u)|_{t=t_{0}} &= (\phi, \psi)
  \end{align*}
  The function $u$ then satisfies a uniform local Strichartz estimate:
  \begin{align*}
    &\norm[e^{(n-2\alpha)(t-t_{0})/q}L^{p}\left( {[t_{0}, t_{0} + T]};
      \Wzero^{1-s,q}(\dk_{t})\right)]{u}  \\
    &\quad + \norm[e^{(n-2\alpha)(t-t_{0})/q}L^{p}\left( {[t_{0}, t_{0} + T]};
      \Wzero^{-s,q}(\dk_{t})\right)]{\pd[t]u} \\
    &\quad\quad\quad\quad\quad\quad\leq C\max\left( T,
      e^{(n-2\alpha)T/2}\right) \norm[{\espace[t_{0}]}]{(\phi, \psi)}
  \end{align*}
  Here the constant is independent of $t_{0}$ and $T$, while $p,q \geq
  2$, $q\neq \infty$, and $(p,q,s)$ satisfies the following:
  \begin{align}
    \label{eq:admissible-epsilon}
    \frac{2}{p} + \frac{n-1}{q} &\leq \frac{n-1}{2} \\ 
    \frac{1}{p} + \frac{n}{q} &= \frac{n}{2} + \epsilon - s \notag
  \end{align}
  
  If, in addition, we know that $h$ is independent of $t$ for large
  $t$, then we may take $\epsilon = 0$ above.
\end{thm}

A similar theorem holds for the inhomogeneous problem:
\begin{thm}
  \label{thm:strichartz-inhomog}
  Suppose that $(X,g)$ is an asymptotically de Sitter space, $t_{0}$
  is sufficiently large, $\lambda > 0$ and $\alpha$ is such that $0
  \leq \alpha < \sqrt{\lambda}$ if $\lambda \leq \frac{n^{2}}{4}$ and
  $\alpha = \frac{n}{2}$ if $\lambda > \frac{n^{2}}{4}$.  Suppose
  further that $\epsilon > 0$ and $u$ satisfies the inhomogeneous
  initial value problem:
  \begin{align*}
    P(\lambda) u &= f\\
    (u,\pd[t]u)|_{t=t_{0}} &= (0,0)
  \end{align*}
  The function $u$ then satisfies a uniform local Strichartz estimate
  with constant independent of $t_{0}$ and $T$:
  \begin{align*}
    &\norm[e^{(n-2\alpha)(t-t_{0})/q}L^{p}\left( {[t_{0}, t_{0} + T]};
      \Wzero^{1-2s,q}(\dk_{t})\right)]{u} \\
    &\quad + \norm[e^{(n-2\alpha)(t-t_{0})/q}L^{p}\left( {[t_{0}, t_{0} + T]};
      \Wzero^{-2s,q}(\dk_{t})\right)]{\pd[t]u} \\
    &\quad\quad\quad\quad\quad\quad\leq C e^{(n-2\alpha)T/2}\max\left(
      1 , T^{\left(
          \frac{1}{q'}-\frac{1}{q}\right)\left(\frac{n-1}{2}\right)}\right)
    \cdot \\
      &\quad\quad\quad\quad\quad\quad\quad \cdot \norm[e^{(n-2\alpha)t/q}L^{p'}\left( {[t_{0}, t_{0}+T]}; L^{q'}(\dk_{t})\right)]{f}
  \end{align*}

  Here $p,q \geq 2$, $q\neq \infty$, and $(p,q,s)$ must satisfy the following:
  \begin{align*}
%    \label{eq:admissible-inhomog}
    \frac{2}{p} + \frac{n-1}{q} &\leq \frac{n-1}{2} \\
    \frac{1}{p} + \frac{n}{q} &= \frac{n}{2} + \epsilon - s \notag
  \end{align*}

  If $h$ is independent of $t$ for large $t$, we may take $\epsilon =
  0$ above.
\end{thm}

The difference in behavior for $\lambda > \frac{n^{2}}{4}$ and
$\lambda \leq \frac{n^{2}}{4}$ is due to the expanding nature of the
spacetime, which prevents energy conservation (or even a constant
global energy bound) for $0 \leq \lambda < \frac{n^{2}}{4}$.  The
non-sharpness of the weight is a defect of our energy estimates and is
due to the need to estimate a pseudodifferential operator (which
requires controlling an inhomogeneous $H^{1}$ norm rather than a
homogeneous one).  This leads to our inability to treat the $\lambda =
0$ case.

The difference in behavior can also be seen in the work of
Yagdjian and Galstian \cite{Yagdjian:2009a} on de Sitter space.  In that
work, the authors obtain an analog of the dispersive estimates we
prove in Section~\ref{sec:estim-prop}, with an improved estimate in
the ``large mass'' setting.  

We believe that the loss of $\epsilon$ derivatives in
Theorems~\ref{thm:strichartz-homog} and \ref{thm:strichartz-inhomog}
is an artifact of our method (which requires regularizing a Fourier
integral operator).  Indeed, the work of Yagdjian and Galstian
\cite{Yagdjian:2009a} on de Sitter space contains no such losses.
Further evidence for this belief is that we obtain the estimates
without a loss when $h$ is independent of $t$ for large $t$.  Removing
the loss in the more general setting requires a nontrivial extension
of the Littlewood-Paley theory and is left to a future paper.

We prove only local estimates because the fundamental solution does
not decay along the light cone, as shown by the author
\cite{Baskin:2010}.  This paper contains a sketch of why this
non-decay should imply the non-existence of a global dispersive
estimate.  We are unsure whether the global Strichartz estimates can
still hold.

The author has previously established \cite{Baskin:2010a} Strichartz
estimates without loss and with better decay for the conformal value
of the Klein-Gordon mass ($\lambda = \frac{n^{2}-1}{4}$).  With this
parameter, global Strichartz estimates are conformally equivalent to
local in time Strichartz estimates for the wave equation on a compact
Lorentzian cylinder.  In this current manuscript we do not recover
those stronger estimates.

The inhomogeneous estimate above is weaker than the homogeneous one,
both in terms of the weight and in terms of the exponents.  The main
difference, however, is that the exponents for the inhomogeneous
problem do not ``decouple'', i.e., the spatial exponents must be dual
to each other.  This is due to the non-static nature of asymptotically
de Sitter spaces.  Indeed, in the case of static spacetimes, the
propagator forms a semigroup and thus the estimates for the
inhomogeneous problem follow from those for the homogeneous problem
(see, for example, the paper of Keel and Tao \cite{Keel:1998}).  As the
spacetimes considered here are non-static, we must use a separate
argument to treat the solution operator for the inhomogeneous problem
directly, leading to the requirement above that the exponents for the
solution and the inhomogeneous term must be dual to each other.

As an application of the Strichartz estimates, we prove a small-data
global existence result for a defocusing semilinear Klein-Gordon
equation.  For $\lambda >\frac{n^{2}}{4}$, we consider the following
semilinear equation:
\begin{align}
  \label{eq:semilinear-main}
  P(\lambda) u + f_{k}(u) &= 0 \\
  \left( u, \pd[t]u\right) |_{t=t_{0}} &= (\phi, \psi) \notag
\end{align}
Here $f_{k}$ is a smooth function and must satisfy five conditions:
\begin{align*}
  &\left| f_{k}(u) \right| \lesssim |u|^{k}  \tag{A1}\\
  &|u|\cdot \left| f_{k}'(u) \right| \sim \left| f_{k}(u)\right| \tag{A2}\\
  &f_{k}(u) - f_{k}'(u) \cdot u  \leq 0 \tag{A3} \\
  &F_{k}(u) = \int_{0}^{u}f_{k}(v)\dv \geq 0 \tag{A4}\\
  &F_{k}(u) \sim |u|^{k+1} \,\text{for large }|u| \tag{A5}
\end{align*}
Conditions (A3) and (A4) are imposed so that energy estimates work in
our favor.  Indeed, (A3) implies that $F_{k}(u) -
\frac{1}{2}f_{k}(u)\cdot u \leq 0$.  Moreover, Assumption (A5) allows
us to control the $L^{k+1}$ norm of solutions in terms of the energy.

We prove the following theorem:
\begin{thm}
  \label{thm:main-slw-critical}
  Suppose $(X,g)$ is an asymptotically de Sitter space with $h$
  independent of $t$ for large $t$.  Suppose further that $\lambda >
  \frac{n^{2}}{4}$ and $k = 1 + \frac{4}{n-1}$.  There is an $\epsilon
  > 0$ so that there is a unique solution $u$ to
  equation~\eqref{eq:semilinear-main} provided the initial data
  satisfy the following smallness condition (here $F_{k}(u) =
  \int_{0}^{v}f_{k}(v)\dv$ is a (positive) antiderivative of $f_{k}$):
  \begin{equation*}
    \norm[H^{1}(Y_{t_{0}})]{\phi} + \norm[L^{2}(Y_{t_{0}})]{\psi} +
    \int_{Y_{t_{0}}}F_{k}(\phi)\dk_{t_{0}} < \epsilon
  \end{equation*}
  In this case the solution $u$ lies in the following $L^{p}$ space:
  \begin{equation*}
    u \in L^{k+1}_{\operatorname{loc}}\left( [t_{0}, \infty); L^{k+1}\left( Y_{t}\right)\right).
  \end{equation*}
\end{thm}

The assumptions on $\lambda$ and $f_{k}$ imply that solutions obey a
global energy bound and that the energy is positive definite, while
the assumption on $h$ allows us to use the estimates without loss.
Even with these somewhat restrictive hypotheses, the standard energy
method does not seem to work here, as the Sobolev inequality
introduces an additional exponential term.  Indeed, unless $p$ is
large enough ($p$ must be at least $\frac{2n}{n-2}$), the inclusion
$W^{1,2}(Y, \dh_{t})\to L^{p}(Y,\dh_{t})$ yields an inclusion of the
same form for $(Y,\dk_{t})$ but with an exponentially growing (in $t$)
bound.

Yagdjian \cite{Yagdjian:2009} studied a similar equation (with an
exponentially decaying function multiplying the nonlinearity) on de
Sitter space.  In that work, the author also considers only the
``large mass'' setting ($\lambda \geq \frac{n^{2}}{4}$) and obtains a
global existence result for a family of nonlinearities.  Our present
work considers a larger family of spacetimes and removes the
exponentially decaying coefficient, but at a cost of restricting our
attention to a single power for the nonlinearity.

The proof of the Strichartz estimates relies on an energy estimate and
a dispersive estimate.  The dispersive estimate is obtained by
analyzing the representation of the fundamental solution found by the
author \cite{Baskin:2010,Baskin:2010b}.  We appeal to the parametrix
construction in those papers because the uniform local estimates do
not follow immediately from a scaling argument, even in the setting
where $h$ is independent of $t$ (except possibly in the $\lambda = 0$
setting).  The estimates for the homogeneous problem ($f=0$) require
only the behavior of the solution operator near the diagonal in
$X\times X$.  Obtaining long-time estimates for the inhomogeneous
problem, however, requires the far-field behavior of the propagator
and so needs most of the parametrix construction given in the author's
previous paper.  The proof of the inhomogeneous estimate is the only
place in this paper where we use the full parametrix.

The main ingredients in the proof of the existence result for the
semilinear equation are an energy estimate and a contraction mapping
argument using the inhomogeneous Strichartz estimate.  We do not prove
existence for a wider range of powers because we do not have an
inhomogeneous $L^{1}L^{2}\to L^{p}L^{q}$ Strichartz estimate.  As
mentioned earlier, this is due to the non-static nature of the
spacetime, so that the exponents do not decouple in the inhomogeneous
problem.

In Sections~\ref{sec:asympt-de-sitt} and \ref{sec:solution-operator}
we review asymptotically de Sitter spaces and briefly describe the
structure of the solution operator for the Klein-Gordon equation on
these spaces.  In Section~\ref{sec:energy-estimates}, we establish the
energy estimates, while in Sections~\ref{sec:disp-estim-lagr} and
\ref{sec:estim-prop} we prove the local dispersive estimates and
discuss the obstruction to a global dispersive estimate.
Section~\ref{sec:strichartz-estimates} proves the Strichartz
estimates, while the final section, Section~\ref{sec:an-appl-semil},
discusses the application to the semilinear problem.

\subsection{Notation}
\label{sec:notation}

Throughout this manuscript, $X$ denotes a compact $(n+1)$-dimensional
manifold with boundary, and $x$ is a boundary defining function on
$X$.  The interior of $X$ is diffeomorphic to $\reals_{t} \times Y$ for a
compact $n$-dimensional manifold $Y$.  We typically use $y$ to denote
coordinates on $Y$ and use $(x,y)$ and $(t,y)$ to denote coordinate
systems on $X$ or its interior.  The functions $x$ and $t$ are related
(for large $t$) by $x = e^{-t}$.  

We consider only large $t$ and small $x$.  Correspondingly, when a
constant is independent of $t_{0}$, we mean that it is independent of
$t_{0}$ provided that $t_{0}$ is bounded away from $-\infty$.

The interior of $X$ is equipped with a Lorentzian metric $g$ having a
prescribed form near future infinity:
\begin{equation*}
  g =  \frac{-\dx^{2} + h(x,y,\dy)}{x^{2}} = -\dt^{2} +
  e^{2t}h(e^{-t}, y , \dy)
\end{equation*}
Here $h$ is a smooth (in $x$) family of Riemannian metrics on $Y$.  At
times, we treat $h$ both as a function of $x$ and as a function of
$t$.

We use $Y_{t}$ to denote individual level sets of the function $t$.
Each $Y_{t}$ is diffeomorphic to $Y$ and inherits an induced
Riemannian metric we call $k_{t}$.  For large $t$, it is related to
$h_{t}$ (the restriction of $h$ to $Y_{t}$) by $k_{t} = e^{2t}h_{t}$.
The volume forms of $h_{t}$ and $k_{t}$ are denoted $\dh_{t}$ and
$\dk_{t}$, respectively, and are related by $\dk_{t} = e^{nt}\dh_{t}$
for large $t$.  We use $h_{x}$ and $k_{x}$ when we wish to emphasize
that $h$ and $k$ are also smooth in $x$.

We denote by $\Box_{g}$ the D'Alembertian of $g$.  In coordinates
$(x,y)$ and $(t,y)$ it is given by the following expressions:
\begin{align}
  \label{eq:box-in-coords}
  \Box_{g} &= (x\pd[x])^{2} - nx\pd[x] +
  \frac{x\pd[x]\sqrt{h_{x}}}{\sqrt{h_{x}}}x\pd[x] + x^{2} \lap_{h_{x}}\\
  &= \pd[t]^{2} + n\pd[t] + \frac{\pd[t]\sqrt{h_{t}}}{\sqrt{h_{t}}}\pd[t] +
  e^{-2t}\lap_{h_{t}} \notag
\end{align}
Here $\lap_{h_{x}}$ (respectively, $\lap_{h_{t}}$) is the Laplacian of
the metric $h_{x}$ (respectively, $h_{t}$) restricted to the level
sets of $x$ or $t$.  We adopt the convention that it has positive
spectrum.  We define $P(\lambda ) = \Box _{g} + \lambda$.

We typically consider the inhomogeneous linear initial value problem:
\begin{align}
  \label{eq:inhomog-cauchy}
  P(\lambda) u &= f \\
  \left( u, \pd[t]u\right) |_{t=t_{0}} &= (\phi, \psi) \notag
\end{align}
The homogeneous ``even'' (or ``odd'') problem refers to the situation
when $f=0$ and $\psi = 0$ (or $\phi = 0$).  The inhomogeneous problem
refers to the situation when both $\phi$ and $\psi$ vanish.  We
typically denote by $U_{v}(t,t_{0})$ the solution operator for the
``odd'' homogeneous problem and by $U_{p}(t,t_{0})$ the solution
operator for the ``even'' problem.

We denote by $\espaces{\mu}$ (defined in
Section~\ref{sec:energy-estimates}) the energy space on $Y_{t}$ with
the following norm:
\begin{equation}
  \norm[{\espaces{\mu}}]{(\phi, \psi)}^{2} = \frac{1}{2}\int_{Y_{t}}
  \left[ \left| \grad_{k_{t}} \phi\right|_{k_{t}}^{2} + \left|
      \psi\right|^{2} + \mu \left| \phi\right|^{2}\right] \dk_{t}
\end{equation}
We denote the space $\espaces{\lambda}$ by $\espace$.

We denote by $\Wzero^{s,p}(\dk_{t})$ the $L^{p}$-based Sobolev space
of order $s$, defined with respect to the metric $k_{t}$ (defined in
Section~\ref{sec:regularization}).  It is equipped with the following
norm:
\begin{equation*}
  \norm[\Wzero^{s,p}(\dk_{t})]{\phi} = \left( \int_{Y_{t}}\left| \left( 1
      + \lap_{k_{t}}\right)^{s/2}\phi\right|^{p} \dk_{t}\right)^{1/p}
\end{equation*}

For a (possibly constant) family of Ban\-ach spaces $Z(t)$, we use the
term $L^{p}([t_{0},T]; Z)$ to denote the space of $Z(t)$-valued
functions $u$ on $[t_{0},T]$ satisfying the following:
\begin{equation*}
  \left( \int_{t_{0}}^{T}\norm[Z(t)]{v(t)}^{p}\dt \right)^{1/p} < \infty
\end{equation*}
We provide a more precise characterization of this space in
Section~\ref{sec:strichartz-estimates}.  

The operators we consider are sums of distributions in three classes:
\begin{equation*}
  \LDzero{0}{\Lambda_{1}},\quad \phgLD{\F}{\LC}, \quad \text{and} \quad
  \phg{\F}{\dblspace}
\end{equation*}
These spaces are recalled briefly in
Section~\ref{sec:solution-operator}.

\section{Asymptotically de Sitter spaces}
\label{sec:asympt-de-sitt}

In this section we describe de Sitter space and define the class of
asymptotically de Sitter spaces we consider.  

De Sitter space is the constant curvature spherically symmetric
solution of the vacuum Einstein equations with a positive cosmological
constant.  It can be realized as the one-sheeted hyperboloid $\{
-X_{0}^{2} + \sum _{i=1}^{n+1}X_{i}^{2} = 1\}$ in $(n+2)$-dimensional
Minkowski space and so is diffeomorphic to $\reals \times
\sphere^{n}$.  It is equipped with coordinates $(\tau, \theta)$,
$\theta \in \sphere^{n}$, given implicitly as follows:
\begin{align*}
  X_{0} &= \sinh \tau \\
  X_{i} &= \theta_{i}\cosh \tau .
\end{align*}
In these coordinates, it inherits an induced metric:
\begin{equation*}
  g_{\operatorname{dS}} = - \dtau^{2} + \left( \cosh \tau \right) ^{2} \dtheta^{2}
\end{equation*}
If we restrict our attention to large $\tau$, de Sitter space
provides a model of a closed but expanding universe.  If $T=e^{-\tau}$
near $\tau = +\infty$, then $\tau = +\infty$ is given by $T=0$ and the
metric takes on the following form:
\begin{equation*}
  g_{\operatorname{dS}} = \frac{-\dT^{2} + \frac{1}{4}\left( 1 +
      T^{2}\right) \dtheta^{2}}{T^{2}}
\end{equation*}
In other words, $(\reals \times \sphere^{n}, g_{\operatorname{dS}})$ is
conformally compact with a spacelike boundary at infinity.

The class of asymptotically de Sitter spaces considered in the current
paper is the same class studied by Vasy \cite{Vasy:2009b}.  In what
follows and in the rest of the paper, $X$ is a compact
$(n+1)$-dimensional manifold with boundary, and $x$ is a boundary
defining function for $X$, i.e., $\pd X = \{ x = 0\}$ and $dx |_{x=0}
\neq 0$.

\begin{defn}
  \label{defn:asymp-de-sitter}
  $(X,g)$ is an asymptotically de Sitter space if $g$ is a Lorentzian
  metric on the interior $X^{\circ}$ of $X$, and, in a collar
  neighborhood $[0,\epsilon)_{x}\times \pd X$ of the boundary, $g$ has
  the following form:
  \begin{equation*}
    g = \frac{-\dx^{2} + h}{x^{2}}
  \end{equation*}
  Here $h$ is a smoothly varying family of symmetric $(0,2)$ tensors
  on $X$, $h|_{\pd X}$ is a section of $T^{*}\pd X \otimes T^{*}\pd X$
  (rather than $T^{*}_{\pd X} X \otimes T^{*}_{\pd X}X$), and $h|_{\pd
  X}$ is a Riemannian metric on $\pd X$.
\end{defn}

\begin{remark}
  \label{rem:joshi-sa-barreto}
  Proposition 2.1 of Joshi and S{\'a} Barreto \cite{Joshi:2000} implies
  that this definition is equivalent to the requirement that, in a
  collar neighborhood $[0,\epsilon)_{x} \times \pd X$ of the boundary,
  $g$ has the following form:
  \begin{equation*}
    g = \frac{ -\dx^{2} + h(x,y,\dy)}{x^{2}}
  \end{equation*}
  Here $h(x,y,\dy)$ is a family of Riemannian metrics on $\pd X$.  
\end{remark}

We further impose two global assumptions:
\begin{itemize}
\item[(B1)] The boundary can be written as a disjoint union $\pd X =
  Y_{+} \cup Y_{-}$, where each $Y_{\pm}$ is a union of connected
  components of $\pd X$.
\item[(B2)] Each nullbicharacteristic (or light ray) $\gamma(t)$ of
  $g$ tends to $Y_{\pm}$ as $t\to \pm \infty$, or vice versa.
\end{itemize}
Taken together, assumptions (B1) and (B2) imply the existence of a
global function $T\in C^{\infty}(X)$ with $T = \pm 1$ on $Y_{\pm}$ and
$dT$ everywhere timelike (see, for example the paper of Geroch
\cite{Geroch:1970}).  (In other words, there is a smooth function with
timelike derivative agreeing with $1-x$ near $Y_{+}$ and $x-1$ near
$Y_{-}$.)  The existence of such a function implies that $X$ is
diffeomorphic to $[-1,1] \times Y_{+}$ and that $(X^{\circ},g)$ is a
globally hyperbolic spacetime.  For convenience we set $Y = Y_{+}$.

As $Y_{+}$ represents future infinity, we typically work with a
timelike foliation of $X^{\circ}$ given by a function $t$ taking
values from $-\infty$ to $\infty$.  We may take this foliation so that
$t = \log x$ near $Y_{-}$ and $t = -\log x$ near $Y_{+}$.  We denote
by $Y_{t}$ the leaves of this foliation, i.e., $Y_{t_{0}} = \{ t =
t_{0}\}$.  We denote by $k_{t}$ the induced metric on $Y_{t}$, while
$h_{t}$ denotes the restriction of the family of Riemannian metrics
$h$ to $Y_{t}$.  The volume forms of $k_{t}$ and $h_{t}$ are denoted
$\dk_{t}$ and $\dh_{t}$, respectively.  Near $Y_{+}$, the metrics are
related by $k_{t} = e^{2t}h_{t}$ and $\dk_{t} = e^{nt}\dh_{t}$.

Near $Y_{+}$, we typically work with coordinates $(x,y)$ or
coordinates $(t,y)$.  The form of $\Box_{g}$ in these coordinates is
recorded in equation~\eqref{eq:box-in-coords}.  We set $P(\lambda) =
\Box _{g} + \lambda$.  This convention is chosen so that $\lambda >
\frac{n^{2}}{4}$ corresponds to positive mass.

\section{Energy estimates}
\label{sec:energy-estimates}

In this section we prove a family of energy estimates for the
equation~\eqref{eq:inhomog-cauchy}.  

For a fixed $\mu$, we define the norm for an energy space on the
spacelike slice $Y_{t}$.
\begin{defn}
  \label{defn:energy-space}
  For initial data $(\phi, \psi) \in C^{\infty}(Y_{t})\times C^{\infty}(Y_{t})$, we
  define its energy norm by the following expression:
  \begin{equation}
    \label{eq:energy-norm}
    \norm[{\espaces{\mu}}]{(\phi, \psi)}^{2} = \frac{1}{2}\int_{Y_{t}}
    \left[ \left| \grad_{k_{t}} \phi\right|_{k_{t}}^{2} + \left|
        \psi\right|^{2} + \mu \left| \phi\right|^{2}\right] \dk_{t}
  \end{equation}
\end{defn}
For $\mu \geq 0$, this is a positive definite norm on the orthogonal
complement of the constant functions, while for $\mu >0$, this is
a positive definite norm.  We define the energy space $\espaces{\mu}$
as the completion of $C^{\infty}(Y_{t})\times C^{\infty}(Y_{t})$ with
respect to this norm.  We denote by $\espace$ the space
$\espaces{\lambda}$.

\begin{remark}
  \label{rem:deficiency}
  Note that for $\lambda = 0$, the $\espace$ norm of $(\phi, \psi)$
  does not control the $L^{2}$ norm of $\phi$.  For this reason we
  restrict our attention to the case when $\lambda > 0$.
\end{remark}

\begin{remark}
  \label{rem:energy-estimates-1}
  We later conjugate our operator by $e^{nt/2}$ to obtain improved
  energy estimates when $\lambda > \frac{n^{2}}{4}$.  This suggests
  that the natural energy spaces to consider bound $\pd[t]u +
  \frac{n}{2}u$ rather than $\pd[t]u$.  However, in order to ensure
  that the energy spaces behave like standard Sobolev spaces under
  interpolation, we do not consider these more natural energy spaces
  at the present time.  
\end{remark}

We soon require the following lemma.
\begin{lem}
  \label{lem:commutator-calc}
  For sufficiently large $t$, the commutator of $\lap_{k_{t}}$ with
  $\pd[t]$ has the following form:
  \begin{equation}
    \label{eq:commutator}
    \left[ \pd[t], \lap_{k_{t}}\right] = -2\lap_{k_{t}} + e^{-3t}Q
  \end{equation}
  Here $Q$ is a smoothly varying (in $x = e^{-t}$) family of second
  order differential operators.  If $h$ is independent of $t$ for
  large $t$, then $Q = 0$.
\end{lem}

\begin{proof}
  We first note that $\lap_{h}$ is a family of second order
  differential operators on $Y$, varying smoothly in $x$ down to
  $x=0$.  In particular, the commutator, $Q = -[\pd[x], \lap_{h_{x}}]$,
  is a second order differential operator on $Y$ also varying smoothly
  in $x$ down to $x=0$.  If $h$ is independent of $x$, then $Q=0$.

  We now use that $\lap_{k_{x}} = x^{2}\lap_{h_{x}}$ and calculate:
  \begin{equation*}
    \left[ x\pd[x], \lap_{k_{x}}\right] = \left[ x\pd[x], x^{2}\lap_{h_{x}}
    \right] = 2\lap_{k_{x}} - x^{3}Q
  \end{equation*}
  Substituting $\pd[t] = -x\pd[x]$ and $x = e^{-t}$ finishes the calculation.
\end{proof}

For $a \geq 0$ and $\mu > 0$, we now prove an energy estimate for an
operator $\tilde{P}(a,\mu)$ given by the following expression near
$Y_{+}$:
\begin{equation*}
  \tilde{P}(a,\mu) = \pd[t]^{2} + \left( a + O(e^{-t})\right)\pd[t] +
  \lap_{k_{t}} + \mu
\end{equation*}
Here $O(e^{-t})$ denotes a smooth function in $e^{-t}$ that is bounded
by $Ce^{-t}$ for some constant $C$.
\begin{prop}
  \label{prop:energy-est-for-p-tilde}
  If $(X,g)$ is an asymptotically de Sitter space, $a \geq 0$, $\mu >
  0$, and $u$ is a smooth function on the interior of $X$, then $u$
  satisfies the following energy estimate:
  \begin{align}
    \label{eq:energy-est-for-p-tilde}
    \norm[\espaces{\mu}]{\left( u, \pd[t]u\right)(t)} &\leq C\left(
      e^{n(t-t_{0})/2}
      \norm[{\espaces[t_{0}]{\mu}}]{(u, \pd[t]u)(t_{0})} \right.\\
    &\quad\quad+
    \left. \int_{t_{0}}^{t}\norm[L^{2}(\dk_{s})]{\tilde{P}(a,\mu)u}e^{n(t-s)/2}\ds\right)
    \notag
  \end{align}
  Here the constant $C$ is independent of $t_{0}$.
\end{prop}

\begin{proof}
  We know from the local theory of hyperbolic equations that the bound
  holds away from infinity (see, for example, the book of Taylor
  \cite{Taylor:1996}).  We must thus only show the bound near infinity
  for initial data on a slice near infinity.  

  For convenience we assume that $u$ is real-valued and we use the
  relationship between the Laplacian and the gradient:
  \begin{equation*}
    \int_{Y_{t}} \left| \grad u (t)\right|^{2}_{k}\dk_{t} =
    \int_{Y_{t}}u(t)\lap_{k_{t}}u(t)\dk_{t}
  \end{equation*}
  We now compute using the form of the operator $\tilde{P}$. 
  \begin{align*}
    \pd[t]\norm[\espaces{\mu}]{(u,\pd[t]u)(t)}^{2} &= \int_{Y_{t}}
    \left( \pd[t]u\right) \left( \pd[t]^{2}u + \lap_{k_{t}}u + \mu
      u\right) \dk_{t} \\
    &\quad + \frac{1}{2}\int_{Y_{t}} \left( \left|
        \grad_{k_{t}}u\right|^{2}_{k_{t}}+ \left| \pd[t]u\right|^{2} + \mu
      \left|u\right|^{2}\right)\left( n +
      \frac{\pd[t]\sqrt{h_{t}}}{\sqrt{h_{t}}}\right)\dk_{t} \\
    &\quad + \frac{1}{2}\int_{Y_{t}}u\left[ \pd[t],
      \lap_{k_{t}}\right] u\dk_{t}
    \\
    &= \left( n + O(e^{-t})\right)\norm[\espaces{\mu}]{(u,\pd[t]u)(t)}
    \\
    &\quad +
    \frac{1}{2}\int_{Y_{t}}u\left[\pd[t],\lap_{k_{t}}\right]u\dk_{t} +
    \int_{Y_{t}}\left( \tilde{P}(a,\mu)u\right)\left(
      \pd[t]u\right)\dk_{t} \\
    &\quad - \int_{Y_{t}}\left( \pd[t]u\right)^{2}\left( a +
      O(e^{-t})\right)\dk_{t}
  \end{align*}

  We now use the calculation in Lemma~\ref{lem:commutator-calc}, the
  positivity of two of the above terms, and the fact that
  $\int_{Y_{t}}e^{-3t}uQu\dk_{t}$ is controlled by the product of
  $Ce^{-t}$ and $\norm[\espaces{\mu}]{(u,\pd[t]u)}$.  After dividing
  by $2\norm[\espaces{\mu}]{(u,\pd[t]u)}$, we may conclude the
  following:
  \begin{align*}
    &\pd[t]\norm[\espaces{\mu}]{(u,\pd[t]u)(t)} \\
    &\quad \leq \left( n +
      O(e^{-t})\right)\norm[\espaces{\mu}]{(u,\pd[t]u)(t)} +
    \norm[L^{2}(Y, \dk_{t})]{\tilde{P}(a,\mu)u}
  \end{align*}
  An application of Gronwall's inequality finishes the proof.
\end{proof}

We now apply the previous proposition to prove an energy estimate for
solutions $u$ of equation~\eqref{eq:inhomog-cauchy}.  
\begin{prop}
  \label{prop:energy-est-main}
  Suppose that $(X,g)$ is an asymptotically de Sitter space and $u$ is
  smooth on the interior of $X$.  If $\lambda > 0$ and $0 \leq \alpha <
  \sqrt{\lambda}$ for $\lambda \leq \frac{n^{2}}{4}$ or $\alpha =
  \frac{n}{2}$ for $\lambda > \frac{n^{2}}{4}$, then $u$ satisfies the
  following estimate:
  \begin{align}
    \label{eq:energy-est-main}
    \norm[\espace]{(u,\pd[t]u)(t)} &\leq C \left(
      e^{(n-2\alpha)(t-t_{0})/2}\norm[{\espace[t_{0}]}]{(u,\pd[t]u)(t_{0})} \right.\\
    &\quad\quad + \left.
      \int_{t_{0}}^{t}e^{(n-2\alpha)(t-s)/2}\norm[L^{2}(Y_{s},\dk_{s})]{P(\lambda)u}\ds\right)
  \end{align}
  Here the constant $C$ is independent of $t_{0}$ if $t_{0}$ is
  bounded away from $-\infty$.
\end{prop}

\begin{remark}
  \label{rem:energy-alpha}
  Note that for $\lambda > \frac{n^{2}}{4}$, $e^{(n-2\alpha)t/2} \equiv 1$.
\end{remark}

\begin{proof}
  Set $v = e^{\alpha t}u$, $a = n-2\alpha$, and $\mu = \lambda -
  \alpha^{2}$, so that $a \geq 0$ and $\mu > 0$.  Note that $v$ then
  satisfies $\tilde{P}(\alpha, \mu)v = e^{\alpha t}P(\lambda) u$.
  Proposition~\ref{prop:energy-est-for-p-tilde} implies the following
  estimate for $v$:
  \begin{align*}
    \norm[\espaces{\mu}]{(v,\pd[t]v)(t)} & \leq C\left(
      e^{n(t-t_{0})/2}\norm[{\espaces[t_{0}]{\mu}}]{(v,\pd[t]v)(t_{0})}\right. \\
    &\quad\quad \left. + \int_{t_{0}}^{t}e^{n(t-s)/2}\norm[L^{2}(Y_{s},
    \dk_{s})]{\tilde{P}(a, \mu)v}\ds\right)
  \end{align*}
  Because $\mu > 0$, we have that
  $\norm[\espaces{\mu}]{(v,\pd[t]v)(t)}$ and $e^{\alpha
    t}\norm[\espace]{(u,\pd[t]u)(t)}$ are comparable.  Dividing
  through by $e^{\alpha t}$ finishes the proof.
\end{proof}

\section{The solution operator}
\label{sec:solution-operator}

In this section we recall from the author's previous work
\cite{Baskin:2010} the description of the solution operator for the
Cauchy problem~\eqref{eq:inhomog-cauchy}.  We use this description
primarily to ensure uniformity in our estimates, as the non-static
nature of the spacetime means that we cannot represent in general the
solution operator as a function of a fixed Laplacian and so must use
another representation of the solution to obtain uniform estimates.

The Schwartz kernel of the solution operator is a Lagrangian
distribution on a compactification $\dblspace$ of the interior of
$X\times X$.  We briefly describe the result below and refer the
reader to the author's previous work for more details
\cite{Baskin:2010,Baskin:2010b}.  Because we are interested only in
obtaining uniform local Strichartz estimates, the homogeneous problem
requires only the behavior near the diagonal near infinity.  We
nevertheless include a fuller description of the Schwartz kernel for
two reasons.  First, it indicates why we limit our attention to
uniform local Strichartz estimates -- there is an
obstruction\footnote{The obstruction is described in
  Section~\ref{sec:an-obstr-glob}.} to a global dispersive estimate
that can be seen from the full parametrix.  Second, the inhomogeneous
problem typically requires an understanding of the behavior of the
solution operator far from the diagonal.  For static spacetimes, one
may rely on the unitary group structure of the propagator to obtain
estimates for the inhomogeneous problem in terms of those for the
homogeneous one.  In fact, this allows the Strichartz exponents for
the inhomogeneous equation to ``decouple'' (so that the source and
target exponents need not be dual to each other).  The non-static
nature of asymptotically de Sitter spaces, however, means that we must
estimate the inhomogeneous solution operator directly.

\subsection{The homogeneous solution operator}
\label{sec:homog-solut-oper}

For the homogeneous solution operators, we restrict our attention to
the behavior of the Schwartz kernel near the diagonal and near
infinity.  In this region, the kernel is most easily described on
$\dblzero$, a manifold with corners compactifying the interior of
$X\times X$.  This is the $0$-double space originally introduced by
Mazzeo and Melrose \cite{Mazzeo:1987}.  

We obtain $\dblzero$ by blowing up the boundary of the diagonal in
$X\times X$, yielding a manifold with corners that has a new boundary
hypersurface, which we call the front face $\frontface$.  The lift of
the diagonal and the flowout of the light cone to this manifold
intersect all boundary hypersurfaces transversely.
Figure~\ref{fig:dblspace} illustrates this blow-up.

\begin{figure}[htp]
  \centering
  \begin{tikzpicture}
    \draw [<->] (-4,3) to (0,3) node[anchor=north](yold){$y,\tilde{y}$};
    \draw [->] (-2,3) to (-2, 4.5) node[anchor=west](xtold){$\tilde{x}$};
    \draw [->] (-2,3) to (-1,2) node[anchor=west](xold) {$x$};
    \draw (-2,3) to (0,4) node[anchor=west](diagold){$\diag$};

    \draw [->, very thick] (0,4.5) .. controls (1, 5) .. (2,4.5);

    \draw [<-] (1,3) to (2,3);
    \draw [->] (4,3) to (5,3) node [anchor=north] (y){$y,\tilde{y}$};
    \draw (2,3) .. controls (3,3.5) .. (4,3) coordinate [pos=0.5,anchor=center] (top);
    \draw (4,3) .. controls (3,2.7) .. (2,3) coordinate [pos=0.5] (bottom);
    \draw [->] (top) to (3, 4.5) node [anchor=west](xt){$\tilde{x}$};
    \draw [->] (bottom) to (4 , 2) node [anchor=west](x){$x$};
    \draw [draw=white,double=black, very thick] (3,3) to (5,4) node [anchor=west](diag){$\diag _{0}$};
  \end{tikzpicture}
  \caption[Passing from $X\times X$ to the $0$-double space.]{Passing
    from $X\times X$ to the $0$-double space $X^2_0 = [X^{2}, \pd\left(\diag\right)
    ]$.}
  \label{fig:dblspace}
\end{figure}

The $0$-double space has three boundary hypersurfaces: $\leftface$,
the lift of the left face of $X\times X$ (given by $x=0$ in $X\times
X$); $\rightface$, the lift of the right face (given by $\xt = 0$ in
$X\times X$); and $\frontface$, the front face introduced by the
blow-up.  Near the front face in a single coordinate chart for $Y$,
the ``polar coordinates'' given below are smooth functions.
\begin{equation*}
  r_{\frontface} = \left( x^{2} + \xt^{2} + \left| y -
      \yt\right|^{2}\right)^{1/2}, \quad \theta = \left( x, y-\yt,
    \xt\right)/r_{\frontface}\in \sphere^{n+1}
\end{equation*}
It is often more convenient to work in projective coordinates near the
front face away from $\rightface$.  These are given by the
$(s,z,\xt,\yt)$, where $s$ and $z$ are the following:
\begin{equation*}
  s = \frac{x}{\xt}, \quad z = \frac{y-\yt}{\xt}
\end{equation*}

The flowout by the Hamilton vector field of the principal symbol of
$P$, $\sigma (P)$, of the characteristic set of $P$ intersected with
the lifted diagonal is a smooth submanifold of $T^{*}\dblzero$ that
intersects all boundary hypersurfaces transversely.  This is a
Lagrangian submanifold of $T^{*}\dblzero$, which we denote
$\Lambda_{1}$.  Near the front face but away from the diagonal,
$\Lambda_{1}$ is the conormal bundle of a submanifold we call the
light cone $\LC$.  Because we are only interested in the region
near the front face, we may assume without loss of generality that
$\LC$ is an embedded submanifold away from the diagonal.

In the following proposition, we summarize the relevant part of the
construction in the author's previous work \cite{Baskin:2010}.  We
denote by $\LDzero[m]{0}{\Lambda_{1}}$ the space of Lagrangian
distributions of order $m$ associated to the flowout Lagrangian
$\Lambda_{1}$ and supported near the diagonal in $\dblzero$.
\begin{prop}
  \label{prop:homog-sol-op}
  Suppose that $t_{0}$ is large and that $U_{p}(t,t_{0})$ and
  $U_{v}(t,t_{0})$ are the solution operators for the ``even'' and
  ``odd'' homogeneous initial value problems, respectively.  For any
  fixed $T_{0}$, let $K_{p}$ and $K_{v}$ be the distributional kernels
  of the operators $U_{p}$ and $U_{v}$, respectively, cut off to the
  region $|t-t_{0}| \leq T_{0}$.  Here $K_{p}$ and $K_{v}$ should be
  regarded as distributions on $X\times X$.  Under these conditions,
  $K_{p}$ and $K_{v}$ lie in the following spaces of Lagrangian
  distributions:
  \begin{align*}
    K_{p} &\in \LDzero[-1/2]{0}{\Lambda_{1}} \\
    K_{v} &\in \LDzero[-3/2]{0}{\Lambda_{1}}
  \end{align*}
\end{prop}

\subsection{The inhomogeneous solution operator}
\label{sec:inhom-solut-oper}

Although we eventually use a representation of the solution of the
inhomogeneous problem in terms of the solution operator of the
homogeneous problem, we do use the extra information contained in this
section.  This is because the iterative process used to extend the
homogeneous estimates in the Section~\ref{sec:homogeneous-problem-1}
does not in general work for the inhomogeneous problem.

We consider the solution operator $E$ for the inhomogeneous problem:
\begin{align*}
  P(\lambda) u &= f\\
  (u,\pd[t]u)|_{t=t_{0}} &= (0,0)
\end{align*}
It can be written in terms of the homogeneous solution operator:
\begin{equation*}
  \left( Ef\right) (t) = \int_{t_{0}} ^{t} U_{v}(t,s)f(s)\ds
\end{equation*}

In the rest of this subsection we provide a fuller description of the
Schwartz kernel of the operator $U_{v}(t,s)$ as $t,s\to +\infty$.
More precisely, we describe the kernel of $U_{v}(t,s)$ considered as a
distribution on $X\times X$.

The Schwartz kernel of $U_{v}(t,s)$ is a Lagrangian distribution on a
compactification $\dblspace$ of the interior of $X\times X$.  It is
obtained from $\dblzero$ by blowing up the intersection of the
projection $\LC$ of $\Lambda_{1}$ with the side faces $\leftface$ and
$\rightface$, introducing a new boundary hypersurface called the light
cone face, or $\newface$.\footnote{The full double space on which it lives is a
  bit more complicated, but we omit these complications here as we are
  only interested in the behavior near future infinity.  We refer the
  reader to \cite{Baskin:2010} for more details.}
Figure~\ref{fig:newdblspace} illustrates the space $\dblspace$.

\begin{figure}[htp]
  \centering
  \begin{tikzpicture}
    \draw [<-] (1,3) to (2,3);
    \draw [->] (4,3) to (5,3) node [anchor=west] (y){$y,\tilde{y}$};
    \draw (2,3) .. controls (3,3.5) .. (4,3) coordinate [pos=0.5,anchor=center] (top)
    coordinate [pos=0.25, anchor=center](leftlc)
    coordinate [pos=0.65, anchor=center](rightlc)
    coordinate [pos=0.2, anchor=center](llbuc)
    coordinate [pos=0.3, anchor=center](lrbuc)
    coordinate [pos=0.6, anchor=center](rlbuc)
    coordinate [pos=0.7, anchor=center](rrbuc);
    \draw (4,3) .. controls (3,2.5) .. (2,3) coordinate [pos=0.45]
    (bottom)
    coordinate [pos=0.25, anchor=center](rightlcb)
    coordinate [pos=0.65, anchor=center] (leftlcb)
    coordinate [pos=0.2,anchor=center] (rrbucb)
    coordinate [pos=0.3,anchor=center] (rlbucb)
    coordinate [pos=0.6,anchor=center] (lrbucb)
    coordinate [pos=0.7, anchor=center] (llbucb);

    \draw [draw=white,double=black,very thick] (3,3) to (5,4) node
    [anchor=west](diag){$\Delta _{0}$};

    \draw (3,3) to (leftlc);
    \draw (3,3) to (rightlcb);
    \draw (3,3) to (rightlc);
    \draw (3,3) to (leftlcb);
    \filldraw[draw = black,fill=white]{(llbuc) .. controls
      +(0.15,-0.05) .. (lrbuc) .. controls +(-0.25,0.65) .. ([shift={(-0.25,
        1)}]lrbuc) .. controls +(-0.075,-0.1) .. ([shift={(-0.25, 1)}]llbuc)
      .. controls +(0,-0.35) .. (llbuc)};
    \filldraw[draw=black,fill=white]{(llbucb) .. controls
      +(0.15,0.05) .. (lrbucb) .. controls +(-0.25,-0.65)
      .. ([shift={(-0.25, -1)}]lrbucb) .. controls +(-0.075, 0.1)
      .. ([shift={(-0.25,-1)}]llbucb) .. controls +(0,0.35) .. (llbucb)};
    \filldraw[draw=black,fill=white]{(rrbuc) .. controls
      +(-0.15,-0.05) .. (rlbuc) .. controls +(0.25, 0.65) .. ([shift={(0.25,
        1)}]rlbuc) .. controls +(0.075, -0.1) .. ([shift={(0.25,1)}]rrbuc)
      .. controls +(0,-0.35) .. (rrbuc)};
    \filldraw[draw=black,fill=white]{(rrbucb) .. controls
      +(-0.15,0.05) .. (rlbucb) .. controls +(0.25, -0.65) .. ([shift={(0.25,
        -1)}]rlbucb) .. controls +(0.075, 0.1) .. ([shift={(0.25,-1)}]rrbucb)
      .. controls +(0,0.35) .. (rrbucb)};

    \draw [->] (top) to (3, 4.5) node [anchor=west](xt){$\tilde{x}$};
    \draw [->] (bottom) to (3.5, 2) node [anchor=east](x){$x$};
  \end{tikzpicture}
  \caption{The double space $\dblspace$ near $\frontface$.}
  \label{fig:newdblspace}
\end{figure}

We now summarize the relevant result of the author's previous work
\cite{Baskin:2010}.  As above, we denote by $\LDzero{0}{\Lambda_{1}}$
the space of Lagrangian distributions of order $m$ associated to the
flowout Lagrangian $\Lambda_{1}$ and supported near the diagonal in
$\dblzero$.  We denote by $\phgLD{\F}{\LC}$ the space of distributions
conormal to $\LC$ whose symbols have polyhomogeneous expansions with
index family $\F$ at the side faces of $\dblspace$.  Finally, we
denote by $\phg{\F}{\dblspace}$ the space of polyhomogeneous conormal
distributions with index family $\F$ on $\dblspace$.  For a precise
definition of $\phgLD{\F}{\LC}$, we refer the reader to
\cite{Baskin:2010}.  For a discussion of polyhomogeneous conormal
distributions, we refer the reader to a paper of Melrose
\cite{Melrose:1992}.
\begin{prop}
  \label{prop:soln-op-odd}
  The kernel of the operator $U_{v}(t,s)$, regarded as a distribution
  on $X\times X$, can be written as a sum $K_{1}+K_{2}+K_{3}$, where
  $K_{i}$ are in the following spaces:
  \begin{align*}
    K_{1} &\in \LDzero[-3/2]{0}{\Lambda_{1}} \\
    K_{2} &\in \phgLD[-3/2]{\F}{\LC}\\
    K_{3} &\in \phg{\F}{\dblspace}
  \end{align*}
  Here $\Lambda_{1}$ is the flowout light cone, and the relevant index
  sets in the family $\F$ are given by the following:
  \begin{align*}
    F_{\newface} &= \left\{ (j,\ell):\ell\leq j, j\in \naturals_{0}
    \right\} \\
    F_{\leftface} &= F_{\rightface} = \left\{ \left(\frac{n}{2} \pm \sqrt{\frac{n^{2}}{4}
        - \lambda} + m, 0\right) : m \in
      \naturals_{0}\right\} \\
    F_{\frontface} &= \left\{ (m,0) : m\in\naturals_{0}\right\}
\end{align*}
\end{prop}
 
\subsection{Regularization}
\label{sec:regularization}

Our proof of the dispersive estimates in the general case rely on
regularizing the solution operator.  In this section we describe that
regularization.

We rely on the notion of semiclassical pseudodifferential operators on
the slices $Y_{t}$.  Here the variable $\xt$ acts as the semiclassical
parameter.  In particular, we consider pseudodifferential operators
with Schwartz kernels given by oscillatory integrals:
\begin{equation}
  \label{eq:fam-psi-do}
  \frac{1}{(2\pi \xt)^{n}}\int_{\reals^{n}}e^{i(y-\yt)\cdot \eta / \xt} a\left( \frac{x}{\xt}
    , \frac{y-\yt}{\xt}, \xt, \yt, \eta\right) \deta
\end{equation}
Here $a$ is a symbol in $\eta$ and smooth in the rest of its
arguments.  We primarily use that powers of $1 + \lap_{k_{t}}$ are of
this form.

We also use the boundedness of these operators on $L^{p}$ spaces.
This is a standard result in semiclassical analysis.
\begin{lem}
  \label{lem:action-on-Lp}
  Suppose that $A_{x}$ is a family of pseudodifferential operators of
  order $-\epsilon$ on $Y$ of the form given by
  equation~\eqref{eq:fam-psi-do}.  If $1 < p < \infty$, then $A_{x}$
  is a bounded operator:
  \begin{equation*}
    A_{x}:L^{p}(Y, \dk_{x}) \to L^{p}(Y; \dk_{x})
  \end{equation*}
  Here the bound is independent of $x$.
\end{lem}

\begin{proof}
  We start by proving the same claim for $Y$ equipped with the volume
  form $\dh_{x}$.

  Using a partition of unity, we write the symbol $a = a_{0} +
  a_{\infty}$, where $a_{0}$ is supported near the zero section and
  $a_{\infty}$ is supported away from $0$.  The function $a_{0}$ is a
  Schwartz function in $\eta$, so we may use standard semiclassical
  results (see, e.g., Koch, Tataru, and Zworski \cite{Koch:2007}) to
  conclude that the bound holds for $a_{0}$.

  For $a_{\infty}$, we appeal to Schur's test.  Indeed, we must bound
  the integral over the left and right factors of $Y$, uniformly in
  $x$.  Because $a_{\infty}$ is supported away from $\eta = 0$, we may
  use the principle of non-stationary phase to bound the following integral:
  \begin{equation*}
    \int_{Y}\int_{\reals^{n}}e^{i(y-y')\cdot \eta / x}a_{\infty}(y,y',\eta)\deta\frac{\dy}{x^{n}}
  \end{equation*}
  Indeed, integrating by parts $n$ times gives a bound of $O(x^{n})$,
  which cancels the factor of $x^{-n}$ in the measure.  A similar
  bound applies to the integral in the other factor, proving the
  claim.

  This shows that $A_{x}$ is bounded $L^{p}(Y, \dh_{x}) \to L^{p}(Y,
  \dh_{x})$.    To prove that it is bounded $L^{p}(Y, \dk_{x} ) \to
  L^{p}(Y, \dk_{x})$, we note only that $A_{x}$ commutes with
  multiplication by $x$.
\end{proof}

We require the following notion of $L^{p}$-based Sobolev spaces, which
interpolate with the $L^{2}$-based Sobolev spaces in the standard way.
\begin{defn}
  \label{defn:sobolev-spaces}
  We define the $\Wzero^{s,p}(\dk_{t})$ norm of a function $\phi \in
  C^{\infty}(Y)$:
  \begin{equation*}
    \norm[\Wzero^{s,p}(\dk_{t})]{\phi} = \left(\int_{Y_{t}}\left| \left(
        1 + \lap_{k_{t}}\right)^{s/2}\phi\right|^{p}\dk_{t} \right)^{1/p}
  \end{equation*}
\end{defn}

We now prove a lemma that allows us to regularize the distributions in
Sections~\ref{sec:homog-solut-oper} and \ref{sec:inhom-solut-oper}.
\begin{lem}
  \label{thm:reg-1}
  Suppose $K\in\LDzero{0}{\Lambda_{1}}$ is a Lagrangian distribution
  of order $m$ associated to $\Lambda_{1}$, and $A_{x}\in \Psi^{k}(Y)$
  is a family of semiclassical pseudodifferential operators on $Y$ of
  order $k$, i.e., whose Schwartz kernels are given as oscillatory
  integrals of the form~\eqref{eq:fam-psi-do} with $a$ a symbol of
  order $k$.  The composition $AK$ is then also a Lagrangian
  distribution:
  \begin{equation*}
    AK \in \LDzero[m+k]{0}{\Lambda_{1}}
  \end{equation*}
  Here we may think of $A$ as acting on the left or the right $Y$
  factor in $X\times X$.  

  An analogous statement holds for $K \in \phgLD{\F}{\Lambda_{1}}$.
\end{lem}

\begin{remark}
  \label{rem:regularization-1}
  Because $k_{x} = x^{-2}h$ near $x=0$, we may put $\left( 1 +
    \lap_{k_{x}}\right)^{k/2}$ in this form.  The $x^{-n}$ in front of the
  oscillatory integral should be interpreted as coming from the
  half-density factors in the operator kernel.
\end{remark}

\begin{proof}
  Away from the front face, we may appeal to standard composition
  results.  Although $A$ is not pseudodifferential on $X$, the
  additional wavefront set of the kernel of $A$ is disjoint from the
  operator wavefront set of $K$.  Indeed, in local coordinates near
  $\pd (X\times X)$, it is contained in the following set:
  \begin{equation*}
    \left\{ (x,y,x,\yt, \xi , 0 , -\xi, 0): (y,\yt) \in
      \operatorname{supp} A_{x}\right\}
  \end{equation*}
  The composition of this set with $\Lambda_{1}$ is empty, so we may
  microlocalize $A$ to be a pseudodifferential operator on $X$ without
  changing the singular structure of $AK$.

  Near the front face, we use a different argument.  Suppose that
  $A_{x}$ is a family of pseudodifferential operators on $Y$ with
  Schwartz kernels as in equation~\eqref{eq:fam-psi-do}.  Given a
  symbol $b(x,\xt, y, \yt, \xi)$ of order $m + \frac{1}{2}$, supported
  away from the side faces of $\dblzero$, we observe the following:
  \begin{equation*}
    A_{x}\left( b(x,\xt,y,\yt,\xi) e^{i(y-\yt)\cdot \xi /x}\right) =
    c(x,\xt,y,\yt,\xi)e^{i(y-\yt)\cdot \xi /x}
  \end{equation*}
  Here $c$ is another symbol of order $k + m + \frac{1}{2}$.  This can
  be seen (as in the book of Grigis and Sj{\"o}strand
  \cite{Grigis:1994}) via a careful application of stationary phase to
  the following integral:
  \begin{equation*}
    \int_{\reals^{n}} \int_{Y_{x}} e^{i(y-y')\cdot (\xi - \eta) /
      x}a(y,y',\eta) b(x,\xt, y', \yt, \xi) \frac{\dy'}{x^{n}}\deta
  \end{equation*}  The second statement follows by a similar argument.
\end{proof}

\section{Dispersive estimates for Lagrangian distributions}
\label{sec:disp-estim-lagr}

In this section we establish a family of uniform dispersive estimates
for distributions in the same class as the wave propagator.  In
particular, we prove that $0$-Lagrangian distributions associated to
$\Lambda_{1}$ and supported near the diagonal obey a dispersive
estimate.  We note that the method we use to prove the dispersive
estimate near the diagonal can be extended to a larger class of
Lagrangian distributions.

\begin{thm}
  \label{lem:dispersive-1}
  Suppose that $K\in \LDzero{0}{\Lambda_{1}}$ for $m =
  -\frac{n}{2}-1-\epsilon$ and $\epsilon > 0$ and that $K$ is
  supported near the diagonal in a small neighborhood of the front
  face.  Then in fact, in terms of coordinates $(s,z,\xt,\yt)$, $K$
  satisfies the following bound:
  \begin{equation*}
    \left| K \right| \lesssim \left| \log s\right| ^{-(n-1)/2 + \epsilon}
  \end{equation*}
  Moreover, because $|\log s |\lesssim 1$ on the support of $K$, we
  may ignore the $\epsilon$ in the previous bound.
\end{thm}

\begin{proof}
  We start by showing the dispersive estimate in the case where the
  Lagrangian $\Lambda_{1}$ is parametrized by the phase function
  $\phi_{0}$, given by the following in terms of $s$ and $z$:
  \begin{equation*}
    \phi_{0} = z\cdot \zeta \pm (1-s)|\zeta|
  \end{equation*}
  The argument is identical for either value of plus or minus, so we
  fix it to be plus.  We assume that the distribution $K$ is supported
  near the front face.  We may thus write it as an oscillatory
  integral of the following form:
  \begin{equation*}
    K = \int _{\reals^{n}} e^{i\phi_{0}}a(s,z,\xt,\yt,\zeta) \dzeta
  \end{equation*}
  Here $a$ is a symbol of order $m+\frac{1}{2}$.  Using polar
  coordinates $\zeta = |\zeta|\hat{\zeta}$ and writing $\phi_{0} =
  (1-s)|\zeta|\left( \frac{z}{1-s}\cdot \hat{\zeta} + 1\right)$, we
  may apply stationary phase to the above oscillatory integral:
  \begin{align*}
    \left| K \right| &\lesssim \int_{(1-s)^{-1}}\left|
      a\left(s,z,\xt,\yt,\pm|\zeta|\hat{z}\right) \right|\cdot \\
    &\quad\quad \left( |\zeta|^{(n-1)/2}(1-s)^{-(n-1)/2} +
      O(|\zeta|^{n-1/2}(1-s)^{-(n-2)/2})\right) \differential{|\zeta|}
    \\
    &+ \int_{0}^{(1-s)^{-1}}\int_{\sphere^{n-1}}\left|
      a(s,z,\xt,\yt,|\zeta|\hat{\zeta})\right|
    |\zeta|^{n-1}\differential{\hat{\zeta}}\differential{|\zeta|} 
  \end{align*}
  The first term is bounded by $C(1-s)^{-(n-1)/2 + \epsilon}$ when $m
  = -\frac{n}{2} - 1 - \epsilon$.  The second term within the
  parentheses is similarly bounded.  We then bound the third term by
  the following expression:
  \begin{equation*}
    C + C\int_{1}^{(1-s)^{-1}}|\zeta|^{n+\frac{1}{2}+m}\differential{|\zeta|}
  \end{equation*}
  This is then bounded by $C(1 + (1-s)^{-(n-1)/2 + \epsilon})$ when $m
  = - \frac{n}{2} -1 - \epsilon$.  Because this piece is supported
  near $s=1$, we may bound it by $(1-s)^{-(n-1)/2+\epsilon}$.

  We now use a perturbation argument to show that this estimate holds
  in a neighborhood of the front face for the Lagrangian
  $\Lambda_{1}$.  Indeed, near the front face, $\Lambda_{1}$ may be
  parametrized by $\phi_{0} + |\zeta| r(s,z,\xt,\yt,\hat{\zeta})$,
  where $r = \xt b$ and $b$ is a smooth function of its arguments.
  For small enough $\xt$, $\phi$ is still a phase functions and its
  critical points (in $\hat{\zeta}$) are close to those of
  $\phi_{0}$.  A similar argument to the one above thus shows that, in
  a small neighborhood of the front face, the same bound holds.

  Finally, because we are assuming that $K$ is supported near the
  diagonal, $(1-s)\sim \log s$ on the support of $K$.  This finishes
  the proof.
\end{proof}

We now prove a similar estimate for the larger class of distributions
used for the inhomogeneous problem.
\begin{lem}
  \label{lem:dispersive-2}
  Suppose that $K \in \phgLD{\F}{\Lambda_{1}}$ with $m = -\frac{n}{2}
  - 1 - \epsilon$ and $\epsilon > 0$.  If $K$ is supported near the
  light cone $\LC$ and $F_{\newface}\geq 0$, then $K$ satisfies the
  following bound when both $x$ and $\xt$ are close to $0$:
  \begin{equation*}
    |K| \lesssim \max\left( 1 , \left| \log \left(
          \frac{x}{\xt}\right)\right| ^{-\frac{n-1}{2} + \epsilon}\right)
  \end{equation*}
  Moreover, we may replace the exponent in the previous equation with
  $-\frac{n-1}{2}$.  
\end{lem}

\begin{proof}
  Near the diagonal, the proof is identical to the proof of
  Lemma~\ref{lem:dispersive-1}.  Away from the diagonal, $\Lambda_{1}$
  is the conormal bundle of an embedded submanifold and the phase
  function for this distribution is a perturbation of the following:
  \begin{equation*}
    \phi_{1}(s,z,\xt,\yt, \eta) = \frac{1}{s}\left( (1-s) -
      |z|\right)\eta 
  \end{equation*}
  Here $s = x/\xt$ and $z = (y-\yt) / \xt$.  The order $m$ is
  sufficiently negative that the symbol of the conormal distribution
  is integrable.  We now simply combine the result near the diagonal
  with the symbol bound given from the order of polyhomogeneity.
\end{proof}

We also record the following trivial but useful observation.
\begin{lem}
  \label{lem:dispersive-3}
  If $K\in\phg{\F}{\dblspace}$ with all index sets greater than or
  equal to $0$, then its Schwartz kernel is smooth on the interior of
  $X\times X$ and uniformly bounded on $X\times X$.
\end{lem}

\section{Estimating the propagator}
\label{sec:estim-prop}

We now seek dispersive estimates for the propagator and for the
solution operator for the inhomogeneous problem.  We typically
regularize the distribution, appeal to the result of the previous
section to obtain a $L^{1}\to L^{\infty}$ type estimate, interpolate
with an energy estimate, and then de-regularize.

\subsection{The homogeneous problem}
\label{sec:homogeneous-problem}

We start by fixing $t_{0}$ large.  As before, we let $U_{v}(t,t_{0})$
denote the solution operator for the ``odd'' problem and
$U_{p}(t,t_{0})$ denote the solution operator for the ``even''
problem.  To simplify the calculation of adjoints,
we consider the operators $\pd[t]U_{v}(t,t_{0})$, $(1 +
\lap_{k_{t}})^{1/2}U_{v}(t,t_{0})$, $\pd[t]U_{p}(t,t_{0})(1+\lap_{k_{t}})^{-1/2}$,
and $(1+\lap_{k_{t}})^{1/2}U_{p}(t,t_{0})(1+\lap_{k_{t}})^{-1/2}$.  All are
elements of $\LDzero[-1/2]{0}{\Lambda_{1}}$, and the energy estimates
provide $L^{2}\to L^{2}$ bounds for them.  We denote by
$U_{\bullet}'(t,t_{0})$ any one of these operators.  We seek a
dispersive estimate for the products
$U_{\bullet}'(t,t_{0})U_{\bullet}'(s,t_{0})^{*}$, where the adjoint
is taken with respect to the $L^{2}$ inner product.

We denote by $\tilde{U}_{\bullet}$ the regularization of
$U'_{\bullet}$ by order $r = \frac{n+1}{4}+\frac{\epsilon}{2}$, i.e.,
$\tilde{U}_{\bullet} = \left( 1 + \lap_{k}\right)^{-r}U'_{\bullet}$.

In order to estimate $\tilde{U}_{\bullet}\tilde{U}_{\bullet}^{*}$, we
use the following lemma.
\begin{lem}
  \label{lem:composition-1}
  Suppose that $s, t>t_{0}$ and $t_{0}$ is sufficiently large.  For
  each $T_{0}$, the restriction of the product
  $\tilde{U}_{\bullet}(t,t_{0})\tilde{U}_{\bullet}(s,t_{0})^{*}$ to
  the neighborhood $|s-t_{0}|\leq T_{0}$ and $|t-t_{0}|\leq T_{0}$ is
  an element of $\LDzero[-\frac{n}{2}-1-\epsilon]{0}{\Lambda_{1}}$,
  uniformly in $t_{0}$.
\end{lem}

\begin{proof}
  The Lagrangian submanifolds corresponding to $\tilde{U}_{\bullet}$
  and $\tilde{U}_{\bullet}^{*}$ intersect transversely in $Y\times Y$,
  so we may follow the proof of H{\"o}rmander (Theorem 4.2.2 of
  \cite{Hormander:1971}) to see that
  $\tilde{U}_{\bullet}\tilde{U}_{\bullet}^{*}$ is a Fourier integral
  operator.  The uniformity follows from the smoothness of the
  distribution up to the front face.

  The compositions have the stated order because restricting to
  $t'=t_{0}$ shifts the order by $\frac{1}{4}$.
\end{proof}

Combining the previous lemma with the results of
Section~\ref{sec:disp-estim-lagr} proves the following corollary.
\begin{prop}
  \label{prop:disp-prop-1}
  For $\epsilon > 0$, $t_{0}$ large and $s,t$ within a fixed distance
  $T_{0}$ of $t_{0}$, the composition
  $\tilde{U}_{\bullet}(t,t_{0})\tilde{U}_{\bullet}(s,t_{0})^{*}$ is a
  bounded operator $L^{1}(\dk_{s})\to L^{\infty}(\dk_{t})$ with the
  following bound:
  \begin{equation*}
    C|t-s|^{-\frac{n-1}{2}+\epsilon}
  \end{equation*}
  Here the bound is independent of $t_{0}$.
\end{prop}

\begin{proof}
  An $L^{1}(\dk_{s})\to L^{\infty}(\dk_{t})$ estimate is equivalent to
  a pointwise bound on $K$, where $K$ is the Schwartz kernel of the
  composition.  The claim then follows in light of
  Lemmas~\ref{lem:dispersive-1} and \ref{lem:composition-1}.  Indeed,
  the bound from Lemma~\ref{lem:dispersive-1} implies the following
  bound for $K$ (considered as an operator)::
  \begin{equation*}
    \norm[L^{\infty}(\dk_{x})]{K\phi} \lesssim \left|
      \log\frac{x}{\xt}\right| ^{-(n-1)/2+\epsilon}\norm[L^{1}(\dk_{\xt})]{\phi}
  \end{equation*}
  Changing coordinates from $(x,y)$ to $(t,y)$ finishes the proof.
\end{proof}

By interpolating with the energy estimates in
Section~\ref{sec:energy-estimates}, we obtain the following family of
dispersive estimates.
\begin{thm}
  \label{thm:homog-dispersive-full}
  Suppose that $\lambda > 0$, $\epsilon > 0$, and that $\alpha$ is
  such that $0 \leq \alpha < \sqrt{\lambda}$ if $\lambda \leq
  \frac{n^{2}}{4}$ and $\alpha = \frac{n}{2}$ if $\lambda >
  \frac{n^{2}}{4}$.  If $t_{0}$ is sufficiently large and $|t-t_{0}| ,
  |s-t_{0}|\leq T_{0}$ for a fixed $T_{0}$, then the composition
  $U_{\bullet}'(t,t_{0})U_{\bullet}'(s,t_{0})$ is a bounded operator
  between Sobolev spaces:
  \begin{equation*}
    \Wzero^{\left(\frac{1}{q'} - \frac{1}{q}\right)\left(\frac{n+1}{4}
      + \frac{\epsilon}{2}\right) , q'} \left(Y_{s},\dk_{s}\right) \to
  \Wzero^{-\left( \frac{1}{q'}-\frac{1}{q}\right)\left( \frac{n+1}{4}
      + \frac{\epsilon}{2}\right), q}\left( Y_{t}, \dk_{t}\right)
  \end{equation*}
  Here $q'$ denotes the conjugate exponent of $q$ and $\Wzero^{r,q}$
  denotes the $L^{q}$-based Sobolev space of order $r$.  As a map
  between these spaces, the operator obeys the following bound with
  $C$ independent of $t_{0}$:
  \begin{equation*}
    Ce^{(n-2\alpha)(t-t_{0})/q}e^{(n-2\alpha)(s-t_{0})/q}|t-s|^{-\left(
      \frac{1}{q'}-\frac{1}{q}\right)\left( \frac{n-1}{2} 
      \right)} 
  \end{equation*}
\end{thm}

\begin{remark}
  \label{rem:homog-disp-what-is-bound}
  If $\lambda > \frac{n^{2}}{4}$, the exponential terms in the bound
  above disappear.
\end{remark}

\begin{proof}
  We may assume that $t,s > t_{0}$.
  Proposition~\ref{prop:energy-est-main} and the boundedness of
  pseudodifferential operators on $L^{2}$-based Sobolev spaces imply
  that $\tilde{U}_{\bullet}(t,t_{0})\tilde{U}_{\bullet}(s,t_{0})^{*}$
  is a bounded operator:
  \begin{equation*}
    \Wzero^{-\frac{n+1}{4}-\frac{\epsilon}{2}, 2} (Y_{s}, \dk_{s}) \to
    \Wzero^{\frac{n+1}{4} + \frac{\epsilon}{2}, 2} (Y_{t}, \dk_{t})
  \end{equation*}
  As a map between these spaces, the operator is bounded by the following:
  \begin{equation*}
    Ce^{(n-2\alpha)(t-t_{0})/2}e^{(n-2\alpha)(s-t_{0})/2}
  \end{equation*}

  We now interpolate with the bounds in
  Proposition~\ref{prop:disp-prop-1} to see that for $q\in
  (2,\infty)$, the composition
  $\tilde{U}_{\bullet}(t,t_{0})\tilde{U}_{\bullet}(s,t_{0})^{*}$ is a
  bounded operator:
  \begin{equation*}
    \Wzero^{-\frac{2}{q}\left( \frac{n+1}{4} +
        \frac{\epsilon}{2}\right), q'}(Y_{s},\dk_{s}) \to \Wzero
    ^{\frac{2}{q}\left( \frac{n+1}{4} + \frac{\epsilon}{2}\right), q}
    (Y_{t}, \dk_{t})
  \end{equation*}
  As such a map, the operator has the following bound:
  \begin{equation*}
    Ce^{(n-2\alpha)(t-t_{0})/q}e^{(n-2\alpha)(s-t_{0})/q}|t-s|^{-\left(
      \frac{1}{q'} - \frac{1}{q}\right) \left( \frac{n-1}{2} - \epsilon\right)}
  \end{equation*}

  We finally remove the regularization with
  Lemma~\ref{lem:action-on-Lp}. By ignoring the $\epsilon$ in the
  above bound, we may combine the losses from the estimate and the
  de-regularization into a single $\epsilon$.
\end{proof}

If $h$ is independent of $t$ for large $t$, we may insert
Littlewood-Paley projectors and remove the loss of regularity via
now-standard arguments.
\begin{thm}
  \label{thm:h-indep-t-dispersive}
  Suppose that $\lambda > 0$, $\epsilon > 0$, and that $\alpha$ is
  such that $0 \leq \alpha < \sqrt{\lambda}$ if $\lambda \leq
  \frac{n^{2}}{4}$ and $\alpha = \frac{n}{2}$ if $\lambda >
  \frac{n^{2}}{4}$.  If $t_{0}$ is sufficiently large and $|t-t_{0}| ,
  |s-t_{0}|\leq T_{0}$ for a fixed $T_{0}$, then the composition
  $U_{\bullet}'(t,t_{0})U_{\bullet}'(s,t_{0})$ is a bounded operator
  between Sobolev spaces:
  \begin{equation*}
    \Wzero ^{\left( \frac{1}{q'}-\frac{1}{q}\right)\left(
        \frac{n+1}{4}\right), q'}(Y_{s},\dk_{s}) \to \Wzero ^{-\left(
        \frac{1}{q'}-\frac{1}{q}\right)\left(\frac{n+1}{4}\right), q}(Y_{t},\dk_{t})
  \end{equation*}
    Here $q'$ denotes the conjugate exponent of $q$.  As a map
  between these spaces, the operator has the following bound with
  $C$ independent of $t_{0}$:
  \begin{equation*}
    Ce^{(n-2\alpha)(t-t_{0})/q}e^{(n-2\alpha)(s-t_{0})/q}|t-s|^{-\left(
      \frac{1}{q'}-\frac{1}{q}\right)\left( \frac{n-1}{2} 
      \right)} 
  \end{equation*}
\end{thm}

\subsubsection{An obstruction to a global dispersive estimate}
\label{sec:an-obstr-glob}

The non-decay of the principal symbol of the propagator near infinity
provides an obstruction to a global dispersive estimate.  In this
section we briefly review how to obtain a global dispersive estimate
on Minkowski space under our framework and then provide a heuristic
argument for the lack of a global dispersive estimate on
asymptotically de Sitter spaces.  The argument we provide here is
heuristic,  but we believe it can be made rigorous by directly
estimating solutions for the model problem on $\reals_{s}^{+}\times
\reals_{z}^{n}$.  

On Minkowski space $\reals\times \reals^{n}$, the propagator is given
as a linear combination of terms of the following form:
\begin{equation*}
  U(t,s,y,\yt) = \int _{\reals^{n}}e^{i\left( y- \yt\right)\cdot \eta
    \pm i(t-s)|\eta|}a(y,\yt,t,s,\eta)\deta
\end{equation*}
Here $a$ is a symbol of order $0$ in $\eta$.  The local dispersive
estimate is known near the diagonal, so we assume here that
$(t,s,y,\yt)$ satisfy $|y-\yt| , |t-s| \geq 1$.  Our goal is to write
$U$ as a conormal distribution associated to the light cone and
parametrized by a single variable.  This is possible because the light
cone is an embedded hypersurface away from the diagonal.  To do this,
we introduce modified polar coordinates in $\eta$ by writing
$(t-s)|\eta| = \sigma$ and $\hat{\eta} = \frac{\eta}{|\eta|}$.
Applying stationary phase in $\hat{\eta}$ shows that $U$ can be
written in the following form:
\begin{equation}
  \label{eq:conormal-written}
  U(t,s,y,\yt) = \int_{\reals}e^{i\rho(t,s,y,\yt)\sigma}\tilde{a}(t,s,y,\yt,\sigma)\dsigma
\end{equation}
Here $\rho$ is a defining function for the light cone and $\tilde{a}$
is a symbol of order $\frac{n-1}{2}$ in $\sigma$ with symbol norms
decaying as $|t-s|^{-\frac{n-1}{2}}$.  In other words, $U$ is written
in the desired form and its symbol exhibits the claimed decay.

Strictly speaking, however, the $W^{\frac{n-1}{2}, 1}\to L^{\infty}$
dispersive estimate does not hold even on Minkowski space, as the
kernel of the propagator is unbounded.  There are several ways to
handle this minor defect.  Three of the most popular include
considering individually the spectral localizations of the solution,
restricting the $L^{1}$-based Sobolev space to a Hardy space, and
enlarging $L^{\infty}$ to a space of functions of bounded mean
oscillation.  Here we take a different approach.  We consider instead
the dispersive estimate satisfied by the regularization of the
propagator by order $\frac{n+1}{2}+\epsilon$.  In Minkowski space, the
kernel of $(1 + \lap_{y})^{-\frac{n+1}{2}-\epsilon} U(t,s,y,\yt)$ has
the same form as in equation~\eqref{eq:conormal-written} but with a
symbol of order $-1-\epsilon$.  The symbol of the regularized
propagator is then integrable and exhibits decay of the requisite
form, i.e., $(1 + \lap_{y})^{-\frac{n+1}{2}-\epsilon} U(t,s)$
satisfies an $L^{1}\to L^{\infty}$ estimate on Minkowski space with
bound $|t-s|^{-\frac{n-1}{2}}$.

A careful reading of Section 12 of the author's previous
work \cite{Baskin:2010} shows that the principal symbol of the
propagator on an asymptotically de Sitter space (away from the
diagonal) approaches a nonzero constant.  Regularizing $U_{v}$ by
order $\frac{n+1}{2}+\epsilon$ and applying the same argument shows
that the regularized propagator has the following form near the
intersection of the front face and the light cone face:
\begin{equation}
  \label{eq:conormal-on-ads}
  \int_{\reals}e^{i\rho \eta}a\left( \rho,s,\frac{z}{|z|},\xt,\yt, \eta\right)\deta
\end{equation}
Here $\rho = \frac{1-s-|z|}{s}$ and $a$ is an elliptic symbol (for
each fixed $s$) of order $-1-\epsilon$ so that $\left(1 + \eta
  ^{2}\right) ^{\frac{1+\epsilon}{2}} a$ approaches a nonzero constant (for
large $\eta$) as $s\to 0$.  This integral is the inverse Fourier
transform of a symbol and so is a distribution conormal to $\rho = 0$.
Treating it as a family of distributions in $s$, the non-decay of the
principal symbol implies that any pointwise bound the family satisfies
should not decay in $s$.  In particular, the regularized propagator
satisfies an $L^{1}\to L^{\infty}$ bound but with a non-decaying
constant.

Upgrading this heuristic argument to a rigorous one (i.e., saturating
the above $L^{1}\to L^{\infty}$ bound) would require several steps.
The first and simplest step would be to show that applying the
regularized propagator to a series of $L^{1}$-normalized bump
functions approximating the delta function indeed yields a solution of
the form~\eqref{eq:conormal-on-ads} for fixed $\xt$ and $\yt$.  The
difficult step is to then show that the
integral~\eqref{eq:conormal-on-ads} in fact does not decay in $s$,
i.e., there are no miraculous cancellations.  For generic elliptic
symbols of order $-1-\epsilon$ this statement is true, but it seems
difficult to rule out decay for all such $a$.

\subsection{The inhomogeneous problem}
\label{sec:inhom-probl}

In this section we prove the dispersive estimates for $U(t,s)$ that
are needed to prove inhomogeneous Strichartz estimates on arbitrarily
long intervals.  Indeed, we use that the solution operator $E$ for the
inhomogeneous problem is related to the homogeneous propagator via
Duhamel's principle:
\begin{equation*}
  (Ef)(t) = \int_{t_{0}} ^{t}U_{v}(t,s)f(s)\ds
\end{equation*}

We prove the following theorem about the operators $\pd[t]U_{v}$ and
$(1 + \lap_{k_{t}})^{1/2}U_{v}$.
\begin{thm}
  \label{thm:inhomog-dispersive}
  Suppose that $\lambda > 0$, $\epsilon > 0$, and $\alpha$ is such
  that $0 \leq \alpha < \sqrt{\lambda}$ for $\lambda \leq
  \frac{n^{2}}{4}$ and $\alpha = \frac{n}{2}$ for $\lambda >
  \frac{n^{2}}{4}$.  If $t_{0}$ is sufficiently large and $t,s \geq
  t_{0}$, then $\pd[t]U_{v}(t,s)$ and $(1 +
  \lap_{k_{t}})^{1/2}U_{v}(t,s)$ are bounded operators between Sobolev
  spaces:
  \begin{equation*}
    L^{q'}(\dk_{s}) \to \Wzero ^{-2s , q}  (\dk_{t})
  \end{equation*}
  Here $q\in (2,\infty)$ and $q'$ denotes the conjugate exponent of
  $q$.  The regularity exponent $s$ is given by the following:
  \begin{equation*}
   2s =\left(1 - \frac{2}{q}\right) \left( \frac{n+1}{2}+
        \epsilon\right)
  \end{equation*}
  As a map between these Sobolev spaces, $\pd[t]U_{v}(t,s)$ and $(1 +
  \lap_{k_{t}})^{1/2}U_{v}(t,s)$ obey the following bound with $C$ independent
  of $t_{0}$:
  \begin{equation*}
    C e^{(n-2\alpha)(t-s)/q} \max \left( |t-s|^{-\left( \frac{1}{q'}
          -\frac{1}{q}\right)\left( \frac{n-1}{2}\right)}, 1\right)
  \end{equation*}
  
  If $h$ is independent of $t$ for large $t$, we may take $\epsilon = 0$.
\end{thm}

\begin{remark}
  \label{rem:inhom-probl-2}
  The regularity exponent on the right side is less than or equal to $1$ so long as
  $q$ satisfies the following:
  \begin{equation*}
    q \leq 2 + \frac{4}{n-1+2\epsilon}
  \end{equation*}
  In particular, if $h$ is independent of $t$ for large $t$, $2s$ is
  no larger than $1$ for $q \leq 2 + \frac{4}{n-1}$.
\end{remark}

\begin{proof}
  We start by letting $\tilde{U}(t,s)$ be $\pd[t]U_{v}(t,s)$ or $(1 +
  \lap_{k_{t}})^{1/2}U_{v}(t,s)$ regularized by order $\frac{n+1}{2} +
  \epsilon$:
  \begin{align*}
    \tilde{U}(t,s) &= \left( 1 + \lap_{k_{t}}\right)^{-\frac{n+1}{2} -
      \epsilon}\pd[t]U_{v}(t,s)
    \quad\quad\quad\quad \text{ or} \\
     \tilde{U}(t,s) &=\left( 1 +
      \lap_{k_{t}}\right)^{-\frac{n+1}{2}-\epsilon}(1+\lap_{k_{t}})^{1/2}U_{v}(t,s)
  \end{align*}
  Lemmas~\ref{lem:dispersive-1}, \ref{lem:dispersive-2}, and
  \ref{lem:dispersive-3} show that $\tilde{U}(t,s)$ is a bounded
  operator $L^{1}(\dk_{s}) \to L^{\infty}(\dk_{t})$ with the following
  bound:
  \begin{equation*}
    C\max \left( | t- s|^{-\frac{n-1}{2} + \epsilon} , 1 \right)
  \end{equation*}
  (Note that we may choose to remove the $\epsilon$ in the above
  bound.)
  
  Moreover, the energy estimates in Section~\ref{sec:energy-estimates}
  show that $\tilde{U}(t,s)$ is a bounded operator between
  $L^{2}$-based Sobolev spaces:
  \begin{equation*}
    L^{2}(\dk_{s}) \to \Wzero ^{\frac{n+1}{2} + \epsilon, 2}(\dk_{t})
  \end{equation*}
  For $0 \leq \alpha < \sqrt{\lambda}$ if $\lambda \leq n^{2}/4$ and
  $\alpha = n/2$ if $\lambda > n^{2}/4$, the operator has a bound of
  the following form:
  \begin{equation*}
    C e^{(n-2\alpha) (t-s)/2}
  \end{equation*}

  Interpolating the above two bounds and then de-regularizing yields
  that $\pd[t]U_{v}(t,s)$ and $(1+\lap_{k_{t}})^{1/2}U_{v}(t,s)$ are
  bounded operators:
  \begin{equation*}
    L^{q'}(\dk_{s}) \to \Wzero ^{ - \left(
        \frac{1}{q'}- \frac{1}{q} \right) \left( \frac{n+1}{2}+
        \epsilon\right) , q}
  \end{equation*}
  It obeys the following bound:
  \begin{equation*}
    C e^{(n-2\alpha)(t-s)/q}\max\left( |t-s|^{-\left(
          \frac{1}{q'}-\frac{1}{q}\right)\left( \frac{n-1}{2} +
          \epsilon\right)}, 1\right)
  \end{equation*}
  Here $q\in (2,\infty)$, $q'$ is the dual exponent to $q$, and
  $\alpha$ is as above.  (Note that $\frac{1}{q'} - \frac{1}{q} = 1-
  \frac{2}{q}$.)

  When $h$ is independent of $t$ for large $t$, we may repeat the
  above argument, replacing the regularizers with fixed
  Littlewood-Paley projectors to avoid the losses.
\end{proof}

\section{Strichartz estimates}
\label{sec:strichartz-estimates}

In this section we prove uniform local Strichartz estimates.  We treat
the homogeneous and inhomogeneous problems for $\lambda > 0$
separately.  

Before proving the estimates, we define the $L^{p}(I ; Z(t))$ spaces.
\begin{defn}
  \label{defn:strichartz-spaces}
  Suppose that $I \subset \reals$ is an interval and $Z(t)$ a family
  of Banach spaces on $I$.  We define the $L^{p}(I; Z(t))$ norm for a
  $Z$-valued function on $I$ by the following:
  \begin{equation*}
    \norm[L^{p}(I; Z)]{v} = \left( \int_{I} \norm[Z(t)]{v(t)}^{p}\dt\right)^{1/p}
  \end{equation*}
  If this norm is finite, we say that $v\in L^{p}(I; Z)$.
\end{defn}

\subsection{The homogeneous problem}
\label{sec:homogeneous-problem-1}

In this section we prove Theorem~\ref{thm:strichartz-homog}.

\begin{remark}
  \label{rem:homogeneous-problem-2}
  Note that if $s \leq 1$, Theorem~\ref{thm:strichartz-homog} yields
  an $L^{p}L^{q}$ estimate for $u$.  
\end{remark}

\begin{proof}
  For convenience, we give the proof for the ``odd'' problem ($\phi =
  0$).  The proof for the full homogeneous problem is nearly
  identical.  We first prove the estimates on an interval of length
  $T_{0}=1$ and then add them up to obtain the full bound.

  We start by defining, for $\psi \in L^{2}(\dk_{t_{0}})$, the
  operator $T_{q}$:
  \begin{equation*}
    T_{q}\psi = e^{-(n-2\alpha)(t-t_{0})/q}U'(t,t_{0})\psi
  \end{equation*}
  In the above we set $U'(t,t_{0})=\pd[t]U_{v}(t,t_{0})$.  Consider
  also the formal adjoint of $T_{q}$ considered as an operator
  $L^{2}(\dk_{t_{0}}) \to L^{\infty}_{t}L^{2}(\dk_{t})$.  For $F \in
  L^{1}_{t}([t_{0}, t_{0}+T_{0}]; L^{2}(\dk_{t}))$, this operator is
  given by the following:
  \begin{equation*}
    T_{q}^{*}F = \int _{t_{0}} ^{t_{0}+T_{0}} e^{-(n-2\alpha)(t-t_{0})/q}U'(s,t_{0})^{*}F(s)\ds
  \end{equation*}
  In particular, the operator $T_{q}T_{q}^{*}$ is given by the following:
  \begin{align*}
    \left( TT^{*}F\right)(t) &= \int_{t_{0}}^{t_{0}+T_{0}}
    e^{-(n-2\alpha)(t+s-2t_{0})/q}U'(t,t_{0})U'(s,t_{0})^{*}F(s)\ds \\
    &= \int_{t_{0}}^{t_{0}+T_{0}}V(t,s)F(s)\ds
  \end{align*}

  Theorem~\ref{thm:homog-dispersive-full} implies that $V(t,s)$ is
  bounded as an operator between the following spaces:
  \begin{equation*}
    \Wzero ^{\left( \frac{1}{q'}-\frac{1}{q}\right) \left(
        \frac{n+1}{4} + \frac{\epsilon}{2}\right), q'}(Y_{s}, \dk_{s})
    \to \Wzero^{-\left( \frac{1}{q'}-\frac{1}{q}\right) \left(
        \frac{n+1}{4} + \frac{\epsilon}{2}\right), q}(Y_{t},\dk_{t})
  \end{equation*}
  It has the bound $C|t-s|^{-\left(
      \frac{1}{q'}-\frac{1}{q}\right)\left( \frac{n-1}{2}\right)}$.  

  We now apply the Hardy-Littlewood-Sobolev method of fractional
  integration (see, for example, the book of Stein
  \cite{Stein:1970}).  In our setting, this states that convolution
  with $|t|^{-\alpha}$ is a bounded operator $L^{p}(\reals) \to
  L^{r}(\reals)$, provided $\frac{1}{p}+\alpha = \frac{1}{r} + 1$.  In
  particular, $TT^{*}$ is a bounded operator:
  \begin{equation*}
    L^{p'}\left( [t_{0}, t_{0} + T_{0}] ; \Wzero^{s,q'}(\dk_{s})\right) \to
    L^{p}\left( [t_{0}, t_{0} + T_{0}]; \Wzero ^{-s, q}(\dk_{t})\right)
  \end{equation*}
  Here $p'$ is the conjugate exponent to $p$, $q'$ the dual exponent
  to $q$, and $p$ and $s$ are given by the following:
  \begin{align*}
    \frac{2}{p} + \frac{n-1}{q} &= \frac{n-1}{2} \\
    \frac{2}{q} \left( \frac{n+1}{4} + \frac{\epsilon}{2}\right) &= 
     \frac{n+1}{4} + \frac{\epsilon}{2} - s
  \end{align*}
  In other words, the theorem holds in the case where we have equality
  in equation~\eqref{eq:admissible-epsilon} (with a slightly
  different $\epsilon$).

  To prove the remaining estimates, we regularize $U'(t,t_{0})$ by
  order $\frac{n}{2} + \frac{\epsilon}{2}$ and then the resulting
  composition has an integrable symbol.  This yields a bound of the
  following form (here $\tilde{U}'$ is this regularization):
  \begin{equation*}
    \norm[L^{1}(\dk_{s})\to
    L^{\infty}(\dk_{t})]{\tilde{U}'(t,t_{0})\tilde{U}'(s,t_{0})^{*}}\leq C
  \end{equation*}
  Interpolating this bound with the energy estimate and then removing
  the regularization shows that, for all $2 < q < \infty$,
  $U'(t,t_{0})U'(s,t_{0})^{*}$ is a bounded operator $\Wzero^{s,q'}\to
  \Wzero^{-s,q}$ with bound $Ce^{(n-2\alpha)(t+s-2t_{0})/q}$ where $s$
  satisfies the following:
  \begin{equation*}
   \frac{n}{q} + \frac{\epsilon}{q} = \frac{n}{2} + \frac{\epsilon}{2} - s
  \end{equation*}
  In particular, $T_{q}T_{q}^{*}$ is then a bounded operator
  $L^{1}W^{s,q'}\to L^{\infty}W^{-s,q}$.  In other words, we have that
  the theorem holds for $p=\infty$.  Interpolating these bounds with
  the previous ones proves half of the theorem in the ``odd'' case for
  intervals of length $T_{0}$.

  To finish the proof in the ``odd'' case on intervals of length
  $T_{0}$, we may replace $U'(t,t_{0}) = \pd[t]U_{v}(t,t_{0})$ with
  $\left( 1 + \lap_{k_{t}}\right)^{1/2} U_{v}(t,t_{0})$ without
  changing the proof.  The estimates in the ``even'' case follow by
  replacing $U_{v}(t,t_{0})$ with $U_{p}(t,t_{0})\left( 1 +
    \lap_{k_{t_{0}}}\right)^{-1/2}$.  

  The estimate on an interval of length $T$ follows by summing the
  estimates on intervals of length $T_{0} = 1$ and using the energy
  estimates of Proposition~\ref{prop:energy-est-main}.  

  The final statement of the theorem follows by repeating the above
  proof with the dispersive estimates of
  Theorem~\ref{thm:h-indep-t-dispersive}.  
\end{proof}

\begin{remark}
  \label{rem:get-global-by-weight}
  By inserting exponential (or even polynomial) weights, it is
  possible to extend the bounds to be global in time, but for weighted
  spaces.  (One replaces $T_{q}$ with $T_{q} = e^{-(n+\alpha +
    \delta)(t-t_{0})/q}U(t,t_{0})$ in the above.)
\end{remark}

\subsection{The inhomogeneous problem}
\label{sec:inhom-probl-1}

In this section we prove Theorem~\ref{thm:strichartz-inhomog}.  Our
approach to the inhomogeneous problem is similar, but complicated by
our inability to ``sum up'' bounds on small intervals without
incurring exponential penalties from the energy estimate when $\lambda
\leq \frac{n^{2}}{4}$ or obtaining a term of the form
$\norm[L^{1}L^{2}]{f}$ on the right side of the estimate.

\begin{remark}
  \label{rem:inhom-probl-3}
  We make several remarks about the statement of
  Theorem~\ref{thm:strichartz-inhomog}.  :
  \begin{enumerate}
  \item The estimates in Theorem~\ref{thm:strichartz-inhomog} are
    weaker than what might be expected from the homogeneous case
    because we must account for the far-field behavior of the
    propagator.
  \item We require the exponents $q$ and $q'$ for the spatial parts of
    the estimate to be dual to each other because the propagator is no
    longer a unitary group and so we cannot write $U(t,s) = U(t,t_{0})
    U(s,t_{0})^{*}$.  This prevents us both from deducing the
    inhomogeneous estimate directly from the homogeneous one and from
    ``decoupling'' the exponents in the inhomogeneous problem.  The
    regularity exponent here looks slightly different from the one in
    the homogeneous setting ($2s$ versus $s$) for the same reason.
  \item If $q < 2 + \frac{4}{n-1}$ we may guarantee that $1-2s \leq 0$
    by choosing $p$ and $\epsilon$ appropriately.  This yields an
    $L^{p'}L^{q'}\to L^{p}L^{q}$ estimate.
  \end{enumerate}
\end{remark}

We require a variant of the Christ-Kiselev lemma (first proved by
Christ and Kiselev \cite{Christ:2001}), which we state now.  The
version we state is slightly different from those in the literature,
but the proof given by Hassell, Tao, and Wunsch \cite{Hassell:2006}
remains valid when $X$ and $Y$ are replaced by smoothing varying
families of Banach spaces $X(t)$ and $Y(t)$.

\begin{lem}[Christ-Kiselev Lemma \cite{Christ:2001}, see
  \cite{Hassell:2006} for this variant]
  \label{lem:christ-kiselev}
  Let $X(t)$ and $Y(t)$ be smoothly varying families of Banach spaces,
  and for all $s,t\in \reals$, let $K(t,s):X(s)\to Y(t)$ be an
  operator-valued kernel from $X(s)$ to $Y(t)$.  Suppose we have the
  following estimate for all $t_{0}\in \reals$, some $1 \leq p < q
  \leq \infty$ and all $f\in L^{p}\left( (-\infty, t_{0}); X(t)\right)$:
  \begin{equation*}
    \norm[L^{q}\left( {[t_{0}, \infty)};Y(t)\right)]{\int_{s <
        t_{0}}K(t,s)f(s)\ds} \leq A\norm[L^{p}\left( \reals; X(t)\right)]{f}
  \end{equation*}
  There is some constant $C$ (depending on $p$ and $q$) so that the
  following estimate holds:
  \begin{equation*}
    \norm[L^{q}(\reals; Y(t))]{\int_{s< t}K(t,s)f(s)\ds} \leq
    CA\norm[L^{p}(\reals; X(t))]{f}
  \end{equation*}
  Moreover, the same type of estimate holds if $\reals$ is
  replaced by a finite interval.
\end{lem}

\begin{proof}[Proof of Theorem~\ref{thm:strichartz-inhomog}]
  The solution operator $E$ for the inhomogeneous problem is given by
  the following:
  \begin{equation*}
    (Ef)(t) = \int_{t_{0}}^{t}U_{v}(t,s)f(s)\ds
  \end{equation*}
  We seek estimates for $\pd[t]E$ and $(1 + \lap_{k_{t}})^{1/2}E$.
  For convenience we use $L$ to denote either $\pd[t]$ or $(1 +
  \lap_{k_{t}})^{1/2}$.
  
  We fix $T>0$ and consider the operator $A$ given by the following integral:
  \begin{equation*}
    (Af)(t) = \int_{t_{0}}^{t_{0}+T}e^{-(n-2\alpha)t/q}LU_{v}(t,s)e^{(n-2\alpha)s/q}f(s)\ds
  \end{equation*}
  The dispersive estimate in Theorem~\ref{thm:inhomog-dispersive} and
  the Hardy-Littlewood-Sobolev method of fractional integration show
  that $A$ is a bounded operator between the following spaces:
  \begin{equation*}
    L^{p'}\left( [t_{0}, t_{0} + T]; L^{q'}(\dk_{t})\right) \to
    L^{r}\left( [t_{0}, t_{0} + T]; \Wzero ^{-2s, q}(\dk_{t})\right)
  \end{equation*}
  Here $\frac{1}{r} + 1 = \frac{1}{p'} + \frac{n-1}{2}\left(
    \frac{1}{q'}- \frac{1}{q}\right)$, $2s = \left( 1 -
    \frac{2}{q}\right)\left( \frac{n+1}{2} + \epsilon\right)$, and the
  bound is $CT^{\frac{n-1}{2}\left( 1 - \frac{2}{q}\right)}$ (and
  independent of $t_{0}$).

  The Christ-Kiselev lemma (Lemma~\ref{lem:christ-kiselev}) then shows
  that the operator $E_{0}$ is bounded as an operator between the same
  spaces, where $E_{0}$ is given by the following:
  \begin{equation*}
    (E_{0}f)(t) = \int_{t_{0}}^{t}e^{-(n-2\alpha)t/q}LU_{v}(t,s)e^{(n-2\alpha)s/q}f(s)\ds
  \end{equation*}

  The operator $LE$ is related to $E_{0}$ in the following way:
  \begin{equation*}
    (LEf)(t) = e^{(n-2\alpha)t/q}LE_{0}\left( e^{-(n-2\alpha)s/q}f(s)\right)(t)
  \end{equation*}
  The operator $LE$ is thus bounded between the following weighted
  spaces with the same bound and with all exponents as above:
  \begin{align*}
    &e^{(n-2\alpha)t/q}L^{p'}\left( [t_{0}, t_{0} + T];
      L^{q'}(\dk_{t})\right) \\
    &\quad \to e^{(n-2\alpha)t/q}L^{r}\left( [t_{0}, t_{0} + T];
      \Wzero^{-2s, q}(\dk_{t})\right)
  \end{align*}
  If we demand that $r = p$, the dual exponent of $p'$, we must have
  that $p$ and $q$ are related:
  \begin{equation*}
    \frac{2}{p} + \frac{n-1}{q} = \frac{n-1}{2}
  \end{equation*}
  
  A similar argument to the one in the proof of
  Theorem~\ref{thm:strichartz-homog} (regularizing by order $n +
  \epsilon$ and then interpolating) finishes the proof of the main
  estimate.  The statement for $h$ independent of $t$ follows by using
  the improved estimates in Theorem~\ref{thm:inhomog-dispersive}.
\end{proof}

\begin{remark}
  \label{sec:inhom-probl-2}
  Note that in the above, if $\lambda > \frac{n^{2}}{4}$, we could in
  fact ``sum up'' the bounds on small intervals by using the better
  energy estimate.  This would allow us to strengthen the bound in this
  case to $C \max(1, T)$, but at a cost of including the $L^{1}L^{2}$
  norm of $f$.
\end{remark}

\section{An application to a semilinear equation}
\label{sec:an-appl-semil}

In this section we prove Theorem~\ref{thm:main-slw-critical}, the
application of the Strichartz estimates to a class of semilinear
Klein-Gordon equations with $\lambda > \frac{n^{2}}{4}$.  In
particular, we assume $h$ is independent of $t$ for large $t$ and
consider equation~\eqref{eq:semilinear-main} with $k= 1 +
\frac{4}{n-1}$.  

We start by proving an energy estimate for solutions of the semilinear
equation~\eqref{eq:semilinear-main}.  Recall that $F_{k}(u) =
\int_{0}^{u}f_{k}(v)\dv$ is a (positive) antiderivative of $f_{k}$.
\begin{prop}
  \label{prop:energy-est-semilinear}
  Suppose that $\lambda > \frac{n^{2}}{4}$ and that $u$ is a solution
  of the semilinear equation~\eqref{eq:semilinear-main}.  There is a
  constant $C$ (independent of $t_{0}$) so that the following energy
  estimate holds:
  \begin{align*} 
    &E(t) = \frac{1}{2}\int_{Y_{t}}\left[ \left| \pd[t]u(t)\right|^{2}
      + \left| \grad_{k_{t}} u(t)\right|^{2}_{k_{t}} +
      \lambda
      |u(t)|^{2} \right]\dk_{t} + \int_{Y_{t}}F_{k}(u(t))\dk_{t}\\
    &\quad\quad\quad\leq C E(t_{0})
  \end{align*}
\end{prop}

\begin{proof}
  Consider the function $v = e^{nt/2}u$.  The function $v$ satisfies
  the following semilinear equation (in the notation of
  Section~\ref{sec:energy-estimates}):
  \begin{align*}
    \tilde{P}\left(0,\lambda - \frac{n^{2}}{4}\right) v = -
    e^{nt/2}f_{k}\left( e^{-nt/2}v\right)
  \end{align*}
  Consider now the energy $\tilde{E}(t)$ of the function $v$:
  \begin{align*}
    \tilde{E}(t) &= \frac{1}{2}\int_{Y_{t}}\left[ \left|
        \pd[t]v\right|^{2} + \left| \grad_{k_{t}}v\right|^{2}_{k_{t}} + \left(
        \lambda - \frac{n^{2}}{4}\right)|v|^{2} \right]\dk_{t}  \\
    &\quad \quad \quad + \int_{Y_{t}}e^{nt}F_{k}\left(
      e^{-nt/2}v\right)\dk_{t}
  \end{align*}
  We now observe that Assumption (A3) above on our nonlinearity
  implies the following:
  \begin{align*}
    &\pd[t] \int_{Y_{t}}e^{nt/2}F_{k}\left( e^{-nt/2}v\right)\dk_{t}
    \\
    &\quad =\int_{Y_{t}}\left(
      e^{nt/2}f_{k}(e^{-nt/2}v)\pd[t]v\right)\dk_{t} \\
    &\quad\quad\quad + \int_{Y_{t}}\left( ne^{nt}F_{k}(e^{-nt/2}v) -
      \frac{n}{2}e^{nt}f_{k}(e^{-nt/2}v)e^{-nt/2}v\right)\dk_{t} \\
    &\quad \leq \int_{Y_{t}} e^{nt/2}f_{k}(e^{-nt/2}v)\pd[t]v \dk_{t}
  \end{align*}
  The methods of Section~\ref{sec:energy-estimates} then imply that
  there is a constant $C$ so that $\tilde{E}(t) \leq C
  e^{nt}\tilde{E}(t_{0})$.

  We now use that $u = e^{-nt/2}v$, together with $\lambda >
  \frac{n^{2}}{4}$ to conclude that there is a constant $C'$
  (independent of $t$) so that the following estimate holds:
  \begin{equation*}
    \frac{1}{C'}e^{nt}E(t) \leq \tilde{E}(t) \leq C' e^{nt}E(t)
  \end{equation*}
  Using the energy bound for $v$ and dividing by $e^{nt}$ finishes the proof.
\end{proof}

\begin{proof}[Proof of Theorem~\ref{thm:main-slw-critical}]
  We start by noting that Theorem~\ref{thm:strichartz-homog}, together with
  $k = 1 + \frac{4}{n-1}$, implies a bound for the solution $u$ of the
  homogeneous problem:
  \begin{align*}
    P(\lambda) u &= 0 \\
    (u, \pd[t]u)|_{t=t_{0}} &= (\phi, \psi)
  \end{align*}
  The bound is the following:
  \begin{equation*}
    \norm[L^{k+1}\left( {[t_{0}, t_{0}+T]};
      L^{k+1}(\dk_{t})\right)]{u} \leq C
    (1+T)\norm[{\espace[t_{0}]}]{(\phi, \psi)}
  \end{equation*}
  Here $C$ is independent of $t_{0}$ and $T$.  Similarly,
  Theorem~\ref{thm:strichartz-inhomog} implies a bound for the
  solution $u$ of the inhomogeneous equation:
  \begin{align*}
    P(\lambda) u &= f, \\
    (u, \pd[t]u)|_{t=t_{0}} &= (0,0)
  \end{align*}
  Indeed, $u$ satisfies the following:
  \begin{equation*}
    \norm[L^{k+1}\left( {[t_{0}, t_{0}+T]};
      L^{k+1}(\dk_{t})\right)]{u} \leq C (1+T^{r})^{1/2}
    \norm[L^{\frac{k+1}{k}}\left( {[t_{0}, t_{0} + T]}; L^{\frac{k+1}{k}}(\dk_{t})\right)]{f}
  \end{equation*}
  Here $r$ is positive and again the constant is independent of
  $t_{0}$ and $T$.  Note that we have used here that $h$ is
  independent of $t$ for large $t$ so that the lossless bounds are
  available.

  Suppose now that $\phi$ and $\psi$ satisfy the following bound:
  \begin{equation*}
    \norm[{\espace[t_{0}]}]{(\phi, \psi)} +
    \norm[L^{k+1}(\dk_{t_{0}})]{\phi} \leq M
  \end{equation*}
  Note that we may control the $L^{k+1}$ norm of $\phi$ by considering
  where $\phi$ is large ($\geq 1$) and small ($\leq 1$).  Assumption
  (A5) implies that the norm where $\phi$ is large is bounded by the
  nonlinear part of the energy while the part where $\phi$ is small is
  bounded by the $L^{2}$ norm of $\phi$ (which is in turn controlled
  by the energy because $\lambda > \frac{n^{2}}{4}$).  This
  observation and the energy estimate in
  Proposition~\ref{prop:energy-est-semilinear} then imply that if $u$
  is a solution of equation~\eqref{eq:semilinear-main} with initial
  data $(\phi, \psi)$, then there is a constant $A$ so that for all $t
  \geq t_{0}$, the following bound holds:
  \begin{equation*}
    \norm[\espace]{(u(t), \pd[t]u(t))} + \norm[L^{k+1}(\dk_{t})]{u(t)}
    \leq AM
  \end{equation*}

  The rest of the proof is now standard but is included here for
  completeness.  We establish a solution on a time interval of length
  $T$ via the contraction mapping principle, and then use the energy
  estimates and the independence of the constants on $t_{0}$ to extend
  the solution.

  We start by finding a solution on an interval $[t_{0}, t_{0}+T]$.
  We seek a fixed point of the following map:
  \begin{equation*}
    \mathcal{F}u(t) = \mathcal{S}(t)(\phi, \psi) + \mathcal{G}(f_{k}(u))(t)
  \end{equation*}
  Here $\mathcal{S}$ is the solution operator for the homogeneous
  problem and $\mathcal{G}$ is the solution operator for the
  inhomogeneous problem with zero initial data.  The main estimate
  used is the following:
  \begin{align*}
    &\norm[L^{\frac{k+1}{k}}\left( {[t_{0},t_{0}+T]};
      L^{\frac{k+1}{k}}(\dk_{t})\right)]{f_{k}(u) -f_{k}(v)} \\
    &\quad\quad\quad \leq C
    \norm[L^{k+1}L^{k+1}]{u-v}\norm[L^{k+1}L^{k+1}]{|u|+|v|}^{k-1}
  \end{align*}
  This estimate follows from the mean value theorem and Assumption
  (A2).  Suppose now that $\phi$ and $\psi$ satisfy the following
  smallness condition:
  \begin{equation*}
    \norm[{\espace[t_{0}]}]{(\phi, \psi)} +
    \norm[L^{k+1}(\dk_{t_{0}})]{\phi} \leq A\epsilon
  \end{equation*}
  The estimates above show that if $\norm[L^{k+1}L^{k+1}]{u} \leq
  K\epsilon$, then $v = \mathcal{F}u$ satisfies the following:
  \begin{equation*}
    \norm[L^{k+1}L^{k+1}]{v} \leq C(1+T^{r})A\epsilon + C' (1+T^{r})^{1/2}(K\epsilon)^{k}
  \end{equation*}
  In particular, if $\epsilon$ is small enough, we may arrange that
  $\norm[L^{k+1}L^{k+1}]{v} \leq K\epsilon$ as well.

  The contraction mapping procedure continues as follows.  Let
  $u^{(-1)}=0$ and $u^{(0)}$ be the solution of the homogeneous
  problem with initial data $(\phi, \psi)$.  For $m> 0$, let $u^{(m)}
  = \mathcal{F}u^{(m-1)}$.  The above estimates then imply that there
  is new constant $C$ so that the following estimate holds:
  \begin{align*}
    &\norm[L^{k+1}\left( {[t_{0},t_{0}+T]};
      L^{k+1}(\dk_{t})\right)]{u^{(m+1)}-u^{(m)}} \leq \\
    &\quad\quad\quad\quad C(1+T^{r})\epsilon\norm[L^{k+1}\left( {[t_{0},t_{0}+T]};
      L^{k+1}(\dk_{t})\right)]{u^{(m)}- u^{(m-1)}}
  \end{align*}
  Thus, we may find $\epsilon$ small (and independent of $t_{0}$) so
  that the sequence $u^{(m)}$ converges to a solution in $L^{k+1}\left(
    [t_{0}, t_{0}+T]; L^{k+1}(\dk_{t})\right)$.  Uniqueness follows by
  standard arguments.  

  We now appeal to the energy bounds extend the solution to an
  interval of length $2T$, as the energy at time $t_{0}+T$ is no
  larger than $A^{2}\epsilon$.  We may thus iterate the process, finishing
  the proof.
\end{proof}

\section{Acknowledgements}
\label{sec:acknowledgements}

The author is very grateful to Rafe Mazzeo and Andr{\'a}s Vasy for
countless helpful discussions and for supervising the thesis on
which much of this research is based.  The author is also grateful
to Jared Wunsch, Austin Ford, Jeremy Marzuola, and Terence Tao for
helpful conversations, and to an anonymous referee for numerous
suggestions to improve this paper.  This research was partially
supported by NSF grants DMS-0801226 and DMS-1103436.

\bibliographystyle{alpha}
\bibliography{papers}

\end{document}